\documentclass[11pt,a4paper]{article}
\usepackage{authblk}
\usepackage{algpseudocode}\usepackage{algorithm}
\usepackage[utf8]{inputenc}
\usepackage[T1]{fontenc}\usepackage{pdfsync}
\usepackage{amsmath}
\usepackage{amsfonts}
\usepackage{amssymb}
\usepackage{amsthm}
\usepackage{xcolor}
\usepackage{bbm,graphicx}
\usepackage{subcaption}
\usepackage{comment}

\usepackage[left=2cm,right=2cm,top=2cm,bottom=2.5cm]{geometry}
\newtheorem{theorem}{Theorem}
\newtheorem{problem}{Problem}
\newtheorem{example}{Example}
\newtheorem{definition}{Definition}

\newtheorem{lemma}{Lemma}
\newtheorem{proposition}{Proposition}
\newtheorem{remark}{Remark}
\newtheorem{assumption}{Assumption}

\providecommand{\norm}[1]{\left\lVert#1\right\rVert}
\providecommand{\abs}[1]{\left\lvert#1\right\rvert}
\providecommand{\pr}[1]{\left(#1\right)} 
\providecommand{\pp}[1]{\left[#1\right]} 
\providecommand{\set}[1]{\left\lbrace#1\right\rbrace} 
\providecommand{\scal}[1]{\left\langle#1\right\rangle}

\def\sign{{\mathop {\rm sign}\nolimits}}
\def\supp{{\mathop {\rm supp}\nolimits}}

\date{}

\begin{document}
\title
{SIR Epidemics With State-Dependent Costs and ICU Constraints: A Hamilton-Jacobi Verification Argument and Dual LP Algorithms}
\author[1,a]{Lorenzo Freddi}
\author[2,3,*,a,b]{Dan Goreac}
\author[2,a,b]{Juan Li}
\author[2,*,a,b]{Boxiang Xu}
\affil[1]{Dipartimento di Scienze Matematiche, Informatiche e Fisiche,  via delle Scienze 206, 33100 Udine, Italy}
\affil[2]{School of Mathematics and Statistics, Shandong University, Weihai, Weihai 264209, PR China}
\affil[3]{LAMA, Univ Gustave Eiffel, UPEM, Univ Paris Est Creteil, CNRS, F-77447 Marne-la-Valle\'ee, France,}
\affil[*]{Corresponding author, email: dan.goreac@univ-eiffel.fr, email: boxiangxu@163.com}
\affil[a]{These authors have contributed equally to this work.}
\affil[b]{The work of these authors has been partially supported by the National Key R and D Program of China (NO. 2018YFA0703901) and the NSF of P.R. China (NOs. 12031009, 11871037).}
\maketitle

\begin{abstract}
The aim of this paper is twofold. On one hand, we strive to give a simpler proof of the optimality of greedy controls when the cost of interventions is control-affine and the dynamics follow a state-constrained controlled SIR model. This is achieved using the Hamilton-Jacobi characterization of the value function, via the verification argument and explicit trajectory-based computations. Aside from providing an alternative to the Pontryagin complex arguments in \cite{AFG2021+} (see also \cite{AFG_2022_corr}), this method allows one to consider more general classes of costs; in particular state-dependent ones. On the other hand,  the paper is completed by linear programming methods allowing one to deal with possibly discontinuous costs. In particular, we propose a brief exposition of classes of linearized dynamic programming principles based on our previous work and ensuing dual linear programming algorithms. We emphasize the particularities of our state space and possible generations of forward scenarios using the description of reachable sets.
\end{abstract}

\section{Introduction}

We deal here with a classical SIR model (cf.\  \cite{kermack1927contribution}) in which the total population is considered to be time-independent and the contact rate is controlled via an external parameter (linked to, say, confinement).  Such models are far from being new  as witnessed by the rich literature on the subject \cite{anderson1992infectious,Behncke,hansen2011optimal,Mart} and the recent pandemic has stimulated an increasing number of contributions, e.g., \cite{alvarez2020simple,Kruse,Ketch,bolzoni2019optimal}. For us, the aim is to maintain the number of infected individuals (or, equivalently, the proportion of infection from the total population) upper-bounded by a chosen parameter $I^*$ (resp., a proportion $i^*$). The parameter is to be thought in connection with the intensive-care unit capacities (see \cite{Kantner,Miclo,AFG2021+}). Motivated by the recent paper \cite{AFG2021+}, we recall the \textit{green} maximal zone in which no control is needed to stabilize the infections and the \textit{yellow} zone in which policies allowing the trajectory to be kept under the level $i^*$ do exist. 
The explicit treatment of the yellow zone (see also \cite{Angulo2021} for geometrical intuitions) is achieved in \cite{AFG2021+} by using normal cone (viability) conditions; see also the recent papers \cite{esterhuizen2021epidemic}, {\cite{ester2021}} for analogous treatments of such sets. We give here some further geometrical insights into the splitting of this yellow zone following the need for \textit{strict confinement} and some conditional viability properties of such zones (see Proposition \ref{PropB} in the Appendix).

In \cite{AFG2021+}, we have considered a state-independent affine-type cost and established the optimality of the greedy policy acting only when the trajectory is about to exit the manageable yellow set. The optimality has been investigated by using an extensive analysis of the Pontryagin extremals of the underlying control problem.  We have chosen here to give a different proof of the optimality of such greedy policies. This is achieved by giving explicit worked-out formulae for the associated cost function, based on computations on the trajectories. It turns out that the regularity of this candidate function allows one to check the naturally-associated Hamilton-Jacobi equation (in an almost sure sense). Furthermore, the method allows to treat considerably more general costs, the price to pay being essentially linked to the usual regularity (cf. Assumption \ref{Ass} and Eq.\ \eqref{Assiv}) and an inequality \eqref{GenCond}. A general class of examples connected with \cite{AFG2021+},  is presented.
We should emphasize that while several authors (e.g.  \cite{Ames2020}, \cite{esterhuizen2021epidemic}) look into such greedy policies, in connection with the "safety" issues,  the conditions on the cost guaranteeing their optimality are not made explicit.

A technical point concerns some hints into the actual regularity of the value function without \eqref{GenCond}.  Since we are focused on a particular type of verification with a specific candidate in mind, we will only present these results in the Appendix A1.  The first discussion, in Section \ref{A1_1}, concerns the uniqueness of the solution of the associated Hamilton-Jacobi equation. We follow the outward qualification and backward invariance arguments in \cite{FrankowskaPlaskacz2000}. The second discussion (cf.\ Section \ref{A1_2}) concerns the inward qualification arguments in \cite{Soner86_1}.

The second part of the paper (Section \ref{Section4}) presents a method allowing one to deal with possibly discontinuous costs. We recall here the embeddings of the state-constrained controlled trajectories into occupation measures, in the spirit of \cite{gaitsgory_quincampoix_09} or \cite{Goreac_Ivascu_13}, \cite{gaitsgory_06}, \cite{Borkar_Gaitsgory_2005}, etc. The state constraints imply a support condition for such measures and the differential formula applied to regular functions of the dynamics imply a linear constraint, where the linearity is in the spirit of functionals of regular measures.  Minimization over such sets of measures is paired with convenient dual formulations (e.g. \cite{gaitsgory_quincampoix_09}, \cite{Serea_Quincampoix_2009}, \cite{G2}, \cite{Goreac_Ivascu_13}). Such formulations are naturally connected with the (subsolutions of the) Hamilton-Jacobi equation satisfied by the value function (e.g., \cite[EQ. 3.1]{gaitsgory_quincampoix_09}, see also our formulation in \eqref{F}).  

Since we work here under state constraints, we present, in the spirit of  \cite{Goreac_Ivascu_13}, a linearized dynamical programming principle (DPP) in \eqref{DPP}. This is complemented by a modified DPP, in \eqref{DDPP}.  For our readers' sake, we sketch the proof,  following Krylov's \textit{shaking of coefficients} method in the Appendix \ref{A2}. This formulation (in \eqref{DDPP}) is very much in the spirit of the dual dynamic programming algorithms introduced in \cite{Pereira1991} and \cite{Shapiro_2011}, also known as Benders decomposition (cf.\ \cite{Benders1962}). Of course, in essence, such algorithms are not completely new, their connections with occupation measures-formulations appearing, for example in \cite{LHT_2008} or,  more recently, in \cite{HWL_2020}. The fundamentals consist in the use of polynomial bases (see \cite{Llavona_1986} for connections with Nachbin's theorem) and Putinar's results on positiveness and moments coherence (see \cite{Putinar1993}). However, the presence of state constraints, the various formulations extending to semi-continuous settings and the Krylov-based duality arguments justify our interest. 
Furthermore, we provide, in the SIR setting, a detailed description of the state-induced constraints on semi-algebraic sets (based on geometrical/viability descriptions) as well as reachability considerations with an impact on the generation of scenarios in the forward stage(s).  To the best of our knowledge, such arguments are new.

Let us briefly explain the structure of the paper.  We begin, in Section \ref{Section2} by recalling the controlled SIR dynamics and basic explicit computations when using constant control policies. We proceed with a description of green (self-regulating, mathematically invariant) and yellow (manageable, mathematically viability kernels) areas with respect to ICU capabilities-induced constraints.  The arguments are illustrated on a toy example where the data correspond more or less to the most recent decision of strict confinement in France retaining an incidence of 400 cases per 100.000 inhabitants, and complemented by the description of some further curves of interest.  The control problem is presented in Section \ref{Section3}. We provide the Hamilton-Jacobi equation associated to the value function as well as a verification argument as to the optimality of a greedy policy.  Explicit computations in connection to such closed-loop policies are provided in Lemma \ref{CompW} which, in turn, allows us to obtain the main verification result, Theorem \ref{ThMain}. The consistency of the assumptions are illustrated in Example \ref{ExpMain}. Section \ref{Section4} deals with the linear formulation of control problems. Relevant results are recalled (set of constraints, primal and dual formulations, equivalence) and they are employed to provide dynamic programming principles (cf.\  Proposition \ref{PropDPP}). The different formulations of two-stage problems are presented and followed by an extensive description of tools and procedures. These are thought in tight connection with the SIR system under analysis. 

Throughout the paper, we denote by ${\mathbb R}_+:=[0,+\infty)$, ${\mathbb R}^*_+:=(0,+\infty)$ and $\mathbb{N}^*=\mathbb{N}\setminus\{0\}$. Given an interval $I\subset\mathbb{R}$ and a (subspace of a) metric space $B$,  $\mathbb{L}^0\pp{I;B}$ will stand for the family of Borel-measurable $B$-valued functions whose domain is $I$. The usual 0/1-valued indicator function of sets will be denoted by $\mathbf{1}$, while the $0/\infty$-version is denoted by $\chi$. 

\section{Controlled SIR systems: geometrical considerations}\label{Section2}
\subsection{The Dynamics}
We consider the following deterministic evolution of a SI(R) model
\begin{equation}\label{SIR}
\begin{cases}
ds^{s_0,i_0,a}(t)=-\beta\pr{1-a(t)}s^{s_0,i_0,a}(t)i^{s_0,i_0,a}(t)\,dt\\
di^{s_0,i_0,a}(t)=\pr{\beta(1-a(t))s^{s_0,i_0,a}(t)-\gamma}i^{s_0,i_0,a}(t)\,dt\\
dr^{s_0,i_0,a}(t)=\gamma i^{s_0,i_0,a}(t)\,dt.
\end{cases}
\end{equation}
The control policy $a$ takes its values in a set $\pp{0,\overline a}$, for some parameter $\overline a\in(0,1)$. For the construction of trajectories, one uses Borel measurable inputs $a\in\mathbb{L}^0\pr{\mathbb{R};\pp{0,\overline a}}$.  The contact rate $\beta>0$ and the recovery rate $\gamma$ are assumed to satisfy the consistency condition $0<\gamma<\beta (1-\overline a)$. As it will be clear afterwards, the proportion $\frac{\gamma}{\beta (1-\overline a)}$ is the one after which strict confinements are to be envisaged. When the spontaneous recovery parameter $\gamma$ is large, the epidemics tend to disappear without strict interventions.
\begin{remark}
1. As it is usually the case, we deal with a time-constant population size. The components $s_0$ and $i_0$ are regarded as proportions i.e. $s_0+i_0\leq 1$. \\
2. A simple look at the dynamics shows that $\mathbb{R}_+\times\mathbb{R}_+$ is time-invariant in the sense that, independently of the control used, starting with $s_0\geq 0$ and $i_0\geq 0$,  neither $s^{s_0,i_0,a}(t)$ nor $i^{s_0,i_0,a}(t)$ becomes negative.\\
3. Actually, in a normalized framework, if $s_0+i_0\leq 1$ (and both non-negative), then $s^{s_0,i_0,a}(t)+i^{s_0,i_0,a}(t)\leq s_0+i_0\leq 1$, for every $t\geq 0$. This yields the time-invariance of the triangle $\mathbb{T}:=\set{\pr{s_0,i_0}\in\mathbb{R}_+^2:\ s_0+i_0\leq 1}$ with respect to the control system \eqref{SIR}.
\end{remark}
One notes that, for every $t\in\mathbb{R}$,  and for constant controls $a\in\pp{0,\overline a}$, \begin{align}\label{Comp0}s^{s_0,i_0,a}(t)+i^{s_0,i_0,a}(t)&=s_0+i_0-\gamma\int_0^ti^{s_0,i_0,a}(u)\,du=s_0+i_0+\frac{\gamma}{\beta(1-a)}\log\frac{s^{s_0,i_0,a}(t)}{s_0}.\end{align}
\subsection{Colour zones}
\begin{definition}
1. We call \textit{green zone} the set $\mathcal{G}$ of initial configurations $\pr{s_0,i_0}$ belonging to the set $\mathbb{T}:=\set{\pr{s_0,i_0}\in\mathbb{R}_+^2:\ s_0+i_0\leq 1}$ for which the infections remain under the level $i^*$ even without any intervention, i.e., 
\[\mathcal{G}:=\set{\begin{split}\pr{s_0,i_0}\in\mathbb{T}:\ \forall a\in\mathbb{L}^0\pr{\mathbb{R};\pp{0,\overline a}},\ \forall t\geq 0, i^{s_0,i_0,a}(t)\leq i^*\end{split}}.\]
2. We call \textit{yellow zone} the set $\mathcal{Y}$ of initial configurations $\pr{s_0,i_0}\in\mathbb{T}$ for which the infections can be kept under $i^*$ with some Borel measurable, $\pp{0,\overline a}$-valued, intervention $a\in\mathbb{L}^0\pr{\mathbb{R}_+;\pp{0,\overline a}}$, i.e., \[\mathcal{Y}:=\set{\begin{split}\pr{s_0,i_0}\in\mathbb{T}:\ \exists\, a\in\mathbb{L}^0\pr{\mathbb{R};\pp{0,\overline a}}\mbox{ s.t.  }
i^{s_0,i_0,a}(t)\leq i^*,\ \forall t\geq 0\end{split}}.\]The strategies $a$ keeping the trajectory under ICU capability will be referred to as {\em admissible}, i.e.,  describing a set $\mathcal{A}d(s_0,i_0)$; when there is no confusion at risk, we will drop the dependency on the initial data.
\end{definition}
The following result is taken from the recent paper \cite{AFG2021+}. Its proof is based on viability (state-constrained evolution) theory and will be omitted here.
\begin{theorem}[Theorem 2.3 in \cite{AFG2021+}]\label{ThViab}
1. The green zone is explicitly given by \begin{align}
\label{Green}
\mathcal{G}=\set{\pr{s_0,i_0}\in\mathbb{T}:\ s_0\leq\phi(i_0)},
\end{align}
where $x:=\phi(i)$ is the unique solution of the equation $-x+\frac{\gamma}{\beta}\log x=i-i^*+\frac{\gamma}{\beta}\log\frac{\gamma}{\beta}-\frac{\gamma}{\beta}$ satisfying $x\geq \frac{\gamma}{\beta}$.\footnote{Another solution of this equation exists but it does not exceed $\frac{\gamma}{\beta}$.}\\
2. The yellow zone is explicitly given by \begin{align}
\label{Yellow}
\mathcal{Y}=\set{\pr{s_0,i_0}\in\mathbb{T}:\ s_0\leq\psi(i_0)},
\end{align}
where $x:=\psi(i)$ is the unique solution of the equation $-x+\frac{\gamma}{\beta\pr{1-\overline a}}\log x=i-i^*+\frac{\gamma}{\beta\pr{1-\overline a}}\log\frac{\gamma}{\beta\pr{1-\overline a}}-\frac{\gamma}{\beta\pr{1-\overline a}}$ satisfying $x\geq \frac{\gamma}{\beta(1-\overline a)}$.\footnote{Another solution of this equation exists but it does not exceed $\frac{\gamma}{\beta(1-\overline a)}$.}\\
\end{theorem}
\begin{example}\label{Exp1} {\em
Let us illustrate this on the example of a total population of $N=67,000,000$ individuals with a maximal tolerance of $400$ cases per $100,000$ habitants (rolling 7 days average) and an incubation-to-recovery time of $t_0=14$ days.  This leads to $I^*=t_0\times\frac{400}{100000}\times N=3,752,000$,\footnote{We are talking here of a level of infected individuals at some moment and not a cumulated number of infections.} i.e.,  a proportion $i^*=0.056$.  We pick $\beta=1/3$ (one infection out of three contacts approximately) and a recovery parameter $\gamma=\frac{1}{14}$ (recovery intervenes in an average of $14$ days). With these data,  cf.\ Figure \ref{Fig1},
\begin{enumerate}
\item a control (confinement) range in $\pp{0,60}$ percent is applied, i.e., $\overline a=0.6$;
\item for the green zone,  
\begin{enumerate}
\item one can start with $I^*$ infections for every $S_0\leq N\frac{\gamma}{\beta}=14,357,143$, that is, a proportion $s_0\leq \frac{3}{14}\simeq 0.214$, leading to the upper-left data tip $\pr{0.214,0.056}$;
\item as $I_0\rightarrow 0$, the green zone extends until $\overline S_{max}:=27,374,986$; equivalently, a proportion $\overline s_{max}\simeq 0.4085819$ leading to the lower-left data tip $\pr{0.409,0+}$;
\end{enumerate}
\item for the yellow zone,
\begin{enumerate}
\item one can start with $I^*$ infections for every $S_0\leq N\frac{\gamma}{\beta\pr{1-\overline a}}=35,892,857$, that is, a proportion $s_0\leq \frac{\gamma}{\beta\pr{1-\overline a}}\simeq 0.5357143$, i.e., the upper-right label $\pr{0.536,0.056}$; 
\item as $I_0\rightarrow 0$, the yellow zone extends until $\underline S_{max}\simeq 54,895,452$; equivalently, a proportion $\underline s_{max}\simeq 0.8193351$ leading to the lower-right data tip $\pr{0.819,0+}$;
\end{enumerate}
\begin{figure}[!t]
\centerline{\includegraphics[width=\columnwidth]{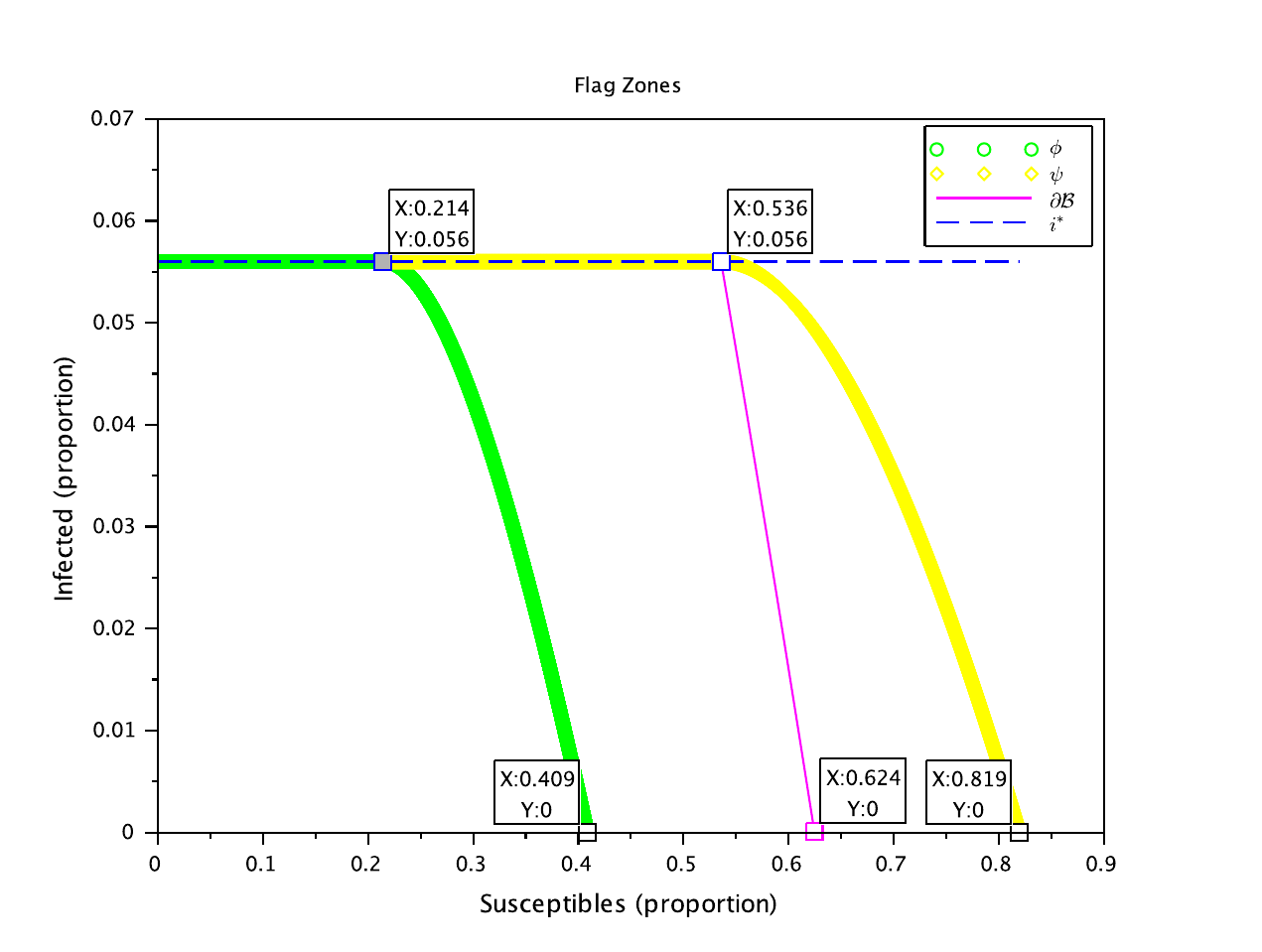}}
\caption{Flag zones for Example \ref{Exp1}: $N=67M$, $I^*=3.75M$, $\beta=\frac{1}{3}$,  $1-\overline a=40\%$, $ \gamma=\frac{1}{14}$.}
\label{Fig1}
\end{figure}
\item finally, with a slightly over-unitary contagion parameter $\beta=1.01$,  we get the following dramatic decrease on safe (green) and yellow (feasible) zones (cf.\ Figure \ref{Fig2})

\begin{figure}[!t]
\centerline{\includegraphics[width=\columnwidth]{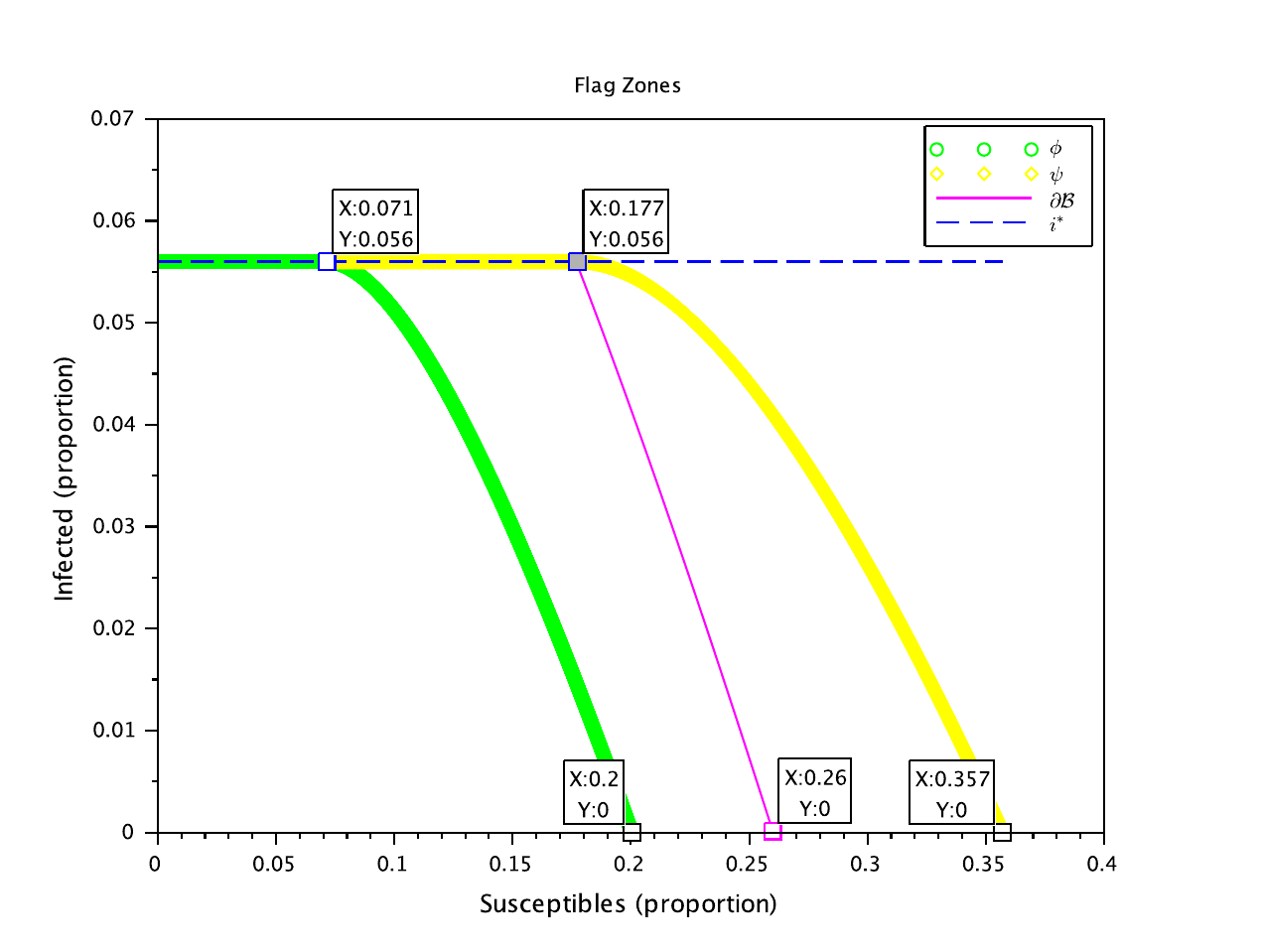}}
\caption{Flag zones for Example \ref{Exp1}: $N=67M$, $I^*=3.75M$, $\beta=1.01$,  $1-\overline a=40\%$, $ \gamma=\frac{1}{14}$.}
\label{Fig2}
\end{figure}
\end{enumerate}}
\end{example}
\subsection{Further geometric considerations}
At this point, the reader may be intrigued by the magenta curve in Figure \ref{Fig1}.  As it turns out, it corresponds to $$\mathbb{R}_+\ni t\mapsto \pr{s^{\frac{\gamma}{\beta(1-\overline a)},i^*,0}(-t),i^{\frac{\gamma}{\beta(1-\overline a)},i^*,0}(-t)},$$ i.e., initial positions from which the $0$-controlled trajectory of \eqref{SIR} reaches the point $\pr{\frac{\gamma}{\beta(1-\overline a)},i^*}$. Since \eqref{Comp0} holds for every $t\in\mathbb{R}$, by reversing time we obtain that an equivalent description of this curve is 
$$\set{\pr{\varphi(i),i}:\ i\in (0,i^*]},$$ where $x:=\varphi(i)\geq\frac{\gamma}{\beta(1-\overline a)}\geq \frac{\gamma}{\beta}$ satisfies $-x+\frac{\gamma}{\beta}\log x=i-i^*+\frac{\gamma}{\beta}\log\frac{\gamma}{\beta(1-\overline a)}-\frac{\gamma}{\beta(1-\overline a)}$.\\
In general,  let us define the set \begin{equation}\label{Bgen}\mathcal{B}^{\overline s}:=\set{\pr{s,i}:\ s\geq 0, \ i\in\pp{0,i^*},\ s\leq \varphi^{\overline s}(i)},
\end{equation}
where $x:=\varphi^{\overline s}(i)\textnormal{ satisfies }-x+\frac{\gamma}{\beta}\log x=i-i^*+\frac{\gamma}{\beta}\log{\overline s}-\overline s$, and  $\sign\pr{x-\frac{\gamma}{\beta}}=\sign\pr{\overline s-\frac{\gamma}{\beta}}$ with 
$\sign(t):=\frac{t}{|t|}$, if $t\neq 0$, $\sign(0):=1$.
In other words, $\varphi^{\overline s}$ is taken on the same side of $\frac{\gamma}{\beta}$ as $\overline s$. 
\begin{remark}\label{RemB}
\begin{enumerate}
\item The reader will have noticed that $\mathcal{B}^{\frac{\gamma}{\beta}}=\mathcal{G}$.
\item Furthermore, to avoid complicated notations, we drop the superscript when $\overline s=\frac{\gamma}{\beta(1-\overline a)}$, i.e., we let $\mathcal{B}:=\mathcal{B}^{\frac{\gamma}{\beta(1-\overline a)}}$. This equally explains the notation employed in Figure \ref{Fig1}. As it will be clear in the verification arguments, the boundary of $\mathcal{B}$ (represented in magenta in Figure \ref{Fig1}) delimits the zone in which strict confinement is required.
\item It is easy to show (please take a look at Proposition \ref{PropB}) that if $\pr{s_0,i_0}\in\mathcal{B}\setminus \mathcal{G}$, then, prior to hitting $i^*$, one has $\pr{s^{s_0,i_0,0}(t),i^{s_0,i_0,0}(t)}\in\mathcal{B}$. Furthermore, due to the maximality of $\mathcal{G}$ (as invariant set), it follows that there exists $t_0\in\mathbb{R}_+$ such that $\pr{s^{s_0,i_0,0}(t_0),i^{s_0,i_0,0}(t_0)}\in\big(\frac{\gamma}{\beta},\frac{\gamma}{\beta(1-\overline a)}\big]\times\set{i^*}$. 
\end{enumerate}
\end{remark}

\section{Optimality of greedy policies}\label{Section3}

\subsection{The control problem}

We now consider the following control problem.

\begin{problem}\label{CtrlProb0}
For every controllable initial configuration $\pr{s_0,i_0}\in\mathcal{Y}$, i.e., subject to $s_0\leq \psi\pr{i_0}$, minimize \begin{equation}\label{Cost0}\begin{split}
J(s_0,i_0,a)
:=\int_{0}^{\tau_{green}^{s_0,i_0,a}}l_1\pr{s^{s_0,i_0,a}(t),i^{s_0,i_0,a}(t),a(t)}dt=\int_0^\infty l_1\pr{s(t),i(t),a(t)}\mathbf{1}_{s(t)\geq \phi\pr{i(t)}}\,dt,
\end{split}\end{equation}
(in the second equation, we have dropped the superscript $(s,i):=\pr{s^{s_0,i_0,a},i^{s_0,i_0,a}}$;)\begin{itemize}
\item under the {\em state constraint} \begin{equation}\label{ICUConstr}
i^{s_0,i_0,a}(t)\leq i^*,\textnormal{ a.s.\ on }\mathbb{R}_+,\end{equation}
\item until reaching the {\em target} $\mathcal{G}$, i.e., until \begin{align}
\label{Target}\nonumber&\tau_{green}^{s_0,i_0,a}:=\inf\set{t\geq 0:\ s^{s_0,i_0,a}(t)< \phi\pr{i^{s_0,i_0,a}(t)}}.
\end{align}
\end{itemize}
\end{problem}
\subsection{The Hamilton-Jacobi equation}
Before proceeding, let us explain the standing assumptions.
\begin{assumption}\label{Ass}
Throughout the rest of the paper, unless otherwise stated, the following assumptions hold true: \begin{itemize}
\item[i.] the cost $l_1$ is non-negative and lower-semicontinuous;  \item[ii.] for every $\pr{s,i}\in\mathbb{R}^2$, the map $\pp{0,\overline a}\ni a\mapsto l_1(s,i,a)$ is non-decreasing;
\item[iii.] the set $\set{l_1(s,i,a):\ a\in\pp{0,\overline a}}$ is convex, for every $\pr{s,i}\in\mathbb{R}^2$. 
\end{itemize}
\end{assumption}
\begin{remark}
1. The reader is invited to note that the second formulation in \eqref{Cost0} is independent of $\tau_{green}$.\\
{2. Please note that the main case of interest is the one in which $l_1$ is control-affine for which Assumption \ref{Ass} iii. is trivially satisfied.}\\
3. Sometimes it is convenient to assume that \begin{equation}\label{Assiv}
l_1(s,i,0)=0 \textnormal{ on }\mathcal{G}.
\end{equation}Then, owing to $a=0$ being optimal as soon as it is admissible (due to Assumption \ref{Ass} ii.), in such framework,  minimizing $J$ is equivalent to minimizing $$\tilde J(s_0,i_0,a):=\int_0^\infty l_1\pr{s^{s_0,i_0,a}(t),i^{s_0,i_0,a}(t),a(t)}dt.$$ The function $\tilde J$ is to be regarded as an infinite horizon discounted value function with discount parameter $q=0$. The arguments developed hereafter equally apply to $q\geq 0$.
\end{remark}

Heuristically speaking, the control Problem \ref{CtrlProb0}\footnote{seen as an infinite horizon $0$-discounted control problem} is connected with the Hamilton-Jacobi equation
\begin{equation}
\label{HJ}
\sup_{a\in\pp{0,\overline a}}\set{\begin{split}\partial_sV(s,i)\beta(1-a)si-\partial_iV(s,i)\pr{\beta(1-a)s-\gamma}i-l_1(s,i,a)\mathbf{1}_{s\geq \phi(i)}\end{split}}= 0,\end{equation}
i.e., provided that $l_1$ is lower-semicontinuous (l.s.c.)\ and bounded from below, $\mathcal{Y}\ni\pr{s,i}\mapsto V(s,i):=\inf_{a\in\mathcal{A}d}J(s,i,a)$ is the smallest l.s.c.\ viscosity super-solution of \eqref{HJ} (with various meanings). For semicontinuity results and connections with the Hamilton-Jacobi equation, the reader is referred to \cite{GS2011}, \cite{FrankowskaPlaskacz2000} (see also \cite{BFZ2011}).  For the continuity of $V$, we refer the reader to  \cite[Remark 5.7 assertion 5]{AFG2021+}.  For the connection with \cite{FrankowskaPlaskacz2000} and the inward pointing condition from \cite{Soner86_1}, we provide some elements in the Appendix \ref{App1}.
\subsection{A verification argument}
We now consider the "greedy" closed-loop feedback control \begin{equation}
\label{bGreed}
a^*\pr{s,i}:=\pp{\left(1-\frac{\gamma}{\beta s}\right)^+\wedge \overline a}\mathbf{1}_{\partial \mathcal{Y}^+}(s,i),
\end{equation}
where, for rigour, we set $\partial \mathcal{Y}^+$ to be the "active" boundary of the viable set $\mathcal{Y}$ i.e.  
\begin{equation}\label{dY+}
\partial \mathcal{Y}^+=\pr{\pp{\frac{\gamma}{\beta},\frac{\gamma}{\beta(1-\overline a)}}\times\set{i^*}}\cup\set{\pr{\psi(i),i}:\ 0<i\leq i^*}.
\end{equation}
\begin{remark}
This control is called greedy as it takes no action until forced (about to exit the yellow set) and, at this time, one takes minimal action guaranteeing that $\mathcal{Y}$ is not exited. {We recall that $a^+=\max\set{a,0}$ stands for the positive part and $a\wedge b=\min\set{a,b}$ is the minimum operator. between two real quantities In the interior of $\mathcal{Y}$ and on the set $\mathcal{G}$, no action ($a=0$) is taken. When the current position is at $\pr{\psi(i),i}$, the trajectory is kept on this boundary $\partial\mathcal{Y}^+$ with the only control available i.e.  $\overline{a}$ until the number of susceptibles reaches $\frac{\gamma}{\beta(1-\overline{a})}$. When the constraint on $i$ is saturated (only manageable for $s\in\big[\frac{\gamma}{\beta},\frac{\gamma}{\beta(1-\overline{a})}\big]$), the deconfinement is progressive ($a=1-\frac{\gamma}{\beta s}$).}
\end{remark}
\begin{lemma}\label{CompW}
Under the assumption \eqref{Assiv}, the function $\mathcal{Y}\cap\pr{\mathbb{R}_+^*}^2\ni\pr{s_0,i_0}\mapsto W(s_0,i_0):=\tilde{J}(s_0,i_0,a^*)$ is explicitly given by 
\begin{eqnarray*}
\textnormal{1)}\ &0, \mbox{ if }\pr{s_0,i_0}\in\mathcal{G},\\
\textnormal{2)}\ &\displaystyle\frac{1}{\gamma i^*}\int_{\frac{\gamma}{\beta}}^{s_1(s_0,i_0)}l_1\pr{u,i^*,1-\frac{\gamma}{\beta u}}du,\mbox{ if }\pr{s_0,i_0}\in\mathcal{B}\setminus\mathcal{G},\\
\textnormal{3)} &\begin{split}&\displaystyle\frac{1}{\beta(1-\overline a)}\int_{\frac{\gamma}{\beta(1-\overline a)}}^{s_2\pr{s_0,i_0}}
\frac{l_1\pr{s,\theta^*-s+\frac{\gamma}{\beta(1-\overline a)}\log s,\overline a}}{s\pr{\theta^*-s+\frac{\gamma}{\beta(1-\overline a)}\log s}}\,ds+W\pr{\frac{\gamma}{\beta(1-\overline a)},i^*},\ \mbox{if }\pr{s_0,i_0}\in\mathcal{Y}\setminus\mathcal{B},\end{split}
\end{eqnarray*}
where \begin{equation}\label{s1s2theta}\begin{cases}
\textnormal{(i)}\ &s_1(s_0,i_0)>\frac{\gamma}{\beta}\textnormal{ is the unique solution to }
s_1-s_0-i_0+i^*-\frac{\gamma}{\beta}\log\frac{s_1}{s_0}=0,\\
\textnormal{(ii)}\ &s_2(s_0,i_0)
:=\exp\pr{\frac{\beta(1-\overline a)}{\gamma\overline a}\pr{s_0+i_0-\frac{\gamma}{\beta}\log s_0-\theta^*}},\\
\textnormal{(iii)}\ &\theta^*:=i^*+\frac{\gamma}{\beta(1-\overline a)}-\frac{\gamma}{\beta(1-\overline a)}\log\frac{\gamma}{\beta(1-\overline a)}.
\end{cases}\end{equation}
\end{lemma}
\begin{proof}
On $\mathcal{G}$, there is nothing to prove. \\
We first deal with the case when $\pr{s_0,i_0}\in \mathcal{B}\setminus\mathcal{G}$. The $0$ controlled trajectory reaches $\partial\mathcal{Y}^+$ at some point $\pr{s_1,i^*}$ (see Remark \ref{RemB}, item 3).  Using the computations on the constant controlled trajectories \eqref{Comp0}, it follows that \[s_1=s_0+i_0-i^*+\frac{\gamma}{\beta}\log\frac{s_1}{s_0}.\]
On the other hand, one explictly computes
$s^{s_1,i^*,a^*}(t)=s_1-\gamma i^*t$ prior to reaching $\frac{\gamma}{\beta}$ (at which point $\mathcal{G}$ has been reached). This leads to the explicit computation (see also Remark \ref{Ass}, 2.)
\begin{equation*}\begin{split}
W(s_0,i_0)&=W(s_1,i^*)=\tilde{J}(s_1,i^*,a^*)=\int_0^{\frac{s_1-\frac{\gamma}{\beta}}{\gamma i^*}}l_1\left(s_1-\gamma i^*t,i^*,1-\frac{\gamma}{\beta \pr{s_1-\gamma i^*t}}\right)dt.
\end{split}\end{equation*} The conclusion follows by using the change of variable $r(t):=s_1-\gamma i^* t$.\\
Let us now consider $\pr{s_0,i_0}\in\mathcal{Y}\setminus\mathcal{B}$. It is clear that one uses $0$ control to reach $\partial \mathcal{Y}$ at some point $\pr{s_2=\psi(i_2),i_2}$ after which, one uses $\overline a$ until reaching $\mathcal{B}$ through the entry point $\big(\frac{\gamma}{\beta(1-\overline a)},i^*\big)$. \\
As in the previous case, using the computations on the constant controlled trajectories \eqref{Comp0}, we have
\[s_2+i_2=s_0+i_0+\frac{\gamma}{\beta}\log\frac{s_2}{s_0}.\]Using the definition of $s_2=\psi(i_2)$(cf.\ Theorem \ref{ThViab}),  it follows that\begin{align*}&s_2+i_2=\frac{\gamma}{\beta(1-\overline a)}\log s_2+i^*+\frac{\gamma}{\beta(1-\overline a)}-\frac{\gamma}{\beta(1-\overline a)}\log\frac{\gamma}{\beta(1-\overline a)}=\frac{\gamma}{\beta(1-\overline a)}\log s_2+\theta^*.\end{align*}Combining the two, we get
\begin{equation}\label{s2i2}
\log s_2=\frac{\beta(1-\overline a)}{\gamma\overline a}\big(s_0+i_0-\frac{\gamma}{\beta}\log s_0-\theta^*\big).
\end{equation}
The $a^*$-controlled trajectory stays on $\partial\mathcal{Y}^+$ and reaches $\pr{\frac{\gamma}{\beta(1-\overline a)},i^*}$ in time $t$. On this set both $s$ and $i$ are monotone. It follows that\begin{align*}
J&:=\int_0^t l_1\pr{s^{s_2,i_2,\overline a}(u),i^{s_2,i_2,\overline a}(u),\overline a}du
=\frac{1}{\beta(1-\overline a)}\int_{\frac{\gamma}{\beta(1-\overline a)}}^{s_2}\frac{l_1\pr{s,\psi^{-1}(s),\overline a}}{s\psi^{-1}(s)}\,ds.
\end{align*}The definition of $\psi$ gives a simple definition for \begin{equation*}\begin{split}\psi^{-1}(s)=&i^*-s+\frac{\gamma}{\beta(1-\overline a)}\log s+\frac{\gamma}{\beta(1-\overline a)}-\frac{\gamma}{\beta(1-\overline a)}\log \frac{\gamma}{\beta(1-\overline a)}=\theta^*-s+\frac{\gamma}{\beta(1-\overline a)}\log s,
\end{split}\end{equation*}
 so that
$$J=\frac{1}{\beta(1-\overline a)}\int_{\frac{\gamma}{\beta(1-\overline a)}}^{s_2}\frac{l_1\pr{s,\theta^*-s+\frac{\gamma}{\beta(1-\overline a)}\log s,\overline a}}{s\pr{\theta^*-s+\frac{\gamma}{\beta(1-\overline a)}\log s}}\,ds.$$
The final expression for $W(s_0,i_0)$ follows then by observing that
\[W(s_0,i_0)=J+W\pr{\frac{\gamma}{\beta(1-\overline a)},i^*}.\]  The lemma is now proved owing to \eqref{s2i2}.
\end{proof}
\begin{remark}
The reader is invited to note that, by definition, $s_2=s^{s_0,i_0,{a^*}}(t_0)$, for some $t_0\geq 0$ such that $s_2(t)\in [\frac{\gamma}{\beta(1-\overline a)},s_0]$ for every $t\in\pp{0,t_0}$. As a consequence, $i^{s_0,i_0,a^*}(t)\geq i_0$ for every $t\in\pp{0,t_0}$. In particular, at $t_0$,
\begin{equation}
\label{estimi2}
\theta^*-s_2+\frac{\gamma}{\beta(1-\overline a)}\log s_2=\psi^{-1}(s_2)=i_2\geq i_0.
\end{equation}
\end{remark}
\begin{theorem}\label{ThMain}
The following propositions hold true.
\begin{enumerate}
\item  $\underset{\mathcal{B}\setminus\mathcal{G}\ni\pr{s_0,i_0}\rightarrow \pr{s',i'}\in\mathcal{G}}{\lim}{s_1(s_0,i_0)=\frac{\gamma}{\beta}}$; moreover, for $\pr{s_0,i_0}\in\mathcal{B}\setminus\mathcal{G}$ we have  $$\partial_{s}s_1(s_0,i_0)=\frac{\pr{\beta s_0-\gamma}s_1}{\pr{\beta s_1-\gamma}s_0},\quad \partial_{i}s_1(s_0,i_0)=\frac{\beta s_1}{\beta s_1-\gamma}.$$
\item $\underset{\mathcal{Y}\setminus\mathcal{B}\ni\pr{s_0,i_0}\rightarrow \pr{s',i'}\in\mathcal{B}}{\lim}{s_2(s_0,i_0)}=\frac{\gamma}{\beta(1-\overline a)}$; moreover, for $\pr{s_0,i_0}\in\mathcal{Y}\setminus\mathcal{B}$ we have \begin{align*}\partial_{s}s_2(s_0,i_0)=\frac{\beta(1-\overline a)}{\gamma \overline a}\pr{1-\frac{\gamma}{\beta s_0}}s_2,\ \partial_{i}s_2(s_0,i_0)=\frac{\beta(1-\overline a)}{\gamma\overline a}s_2.\end{align*} 
\item Under the Assumption \ref{Ass} and assuming \eqref{Assiv} to hold true and $l_1$ to be continuous, the function $W$ is continuous and differentiable almost everywhere on $\mathcal{Y}$.\footnote{Since the equation \eqref{HJ} usually concerns an equality condition on the interior $\overset{\circ}{\mathcal{Y}}$ and an inequality on the boundary $\partial\mathcal{Y}$, the "almost everywhere" here is intended with respect to the Lebesgue measure on $\mathbb{R}^2$ for the interior, respectively a suitable 1-dimensional measure on the regular curve describing $\partial \mathcal{Y}$.}
\item Let us introduce the function $\pr{\mathbb{R}_+^*}^2\times\left(0,\overline a\right]\ni(s,i,a)\mapsto \tilde{l}_1(s,i,a):=\frac{l_1(s,i,a)}{\gamma ia}$ and assume that \begin{align}\label{GenCond}\begin{cases}
\tilde{l}_1\pr{s_1(s_0,i_0),i^*,1-\frac{\gamma}{\beta s_1(s_0,i_0)}}&\leq  \underset{{a\in(0,\overline a]}}{\min}\tilde{l}_1(s_0,i_0,a),\ \mbox{if }\pr{s_0,i_0}\in\mathcal{B}\setminus\mathcal{G},\\
\tilde{l}_1\pr{s_2,\theta^*-s_2+\frac{\gamma}{\beta (1-\overline a)}\log s_2,\overline a}
&\leq \underset{{a\in (0,\overline a]}}{\min}\tilde{l}_1(s_0,i_0,a),\ \mbox{if }\pr{s_0,i_0}\in\mathcal{Y}\setminus\mathcal{B},\\
&\textnormal{where }s_2=s_2(s_0,i_0).
\end{cases}\end{align}Then $W$ satisfies the Hamilton-Jacobi equation \eqref{HJ} at all points where it is differentiable (hence almost everywhere). 
\end{enumerate}
\end{theorem}
\begin{proof}
The first two assertions are simple consequences of the definitions in \eqref{s1s2theta} and those of the sets $\mathcal{G}$ and $\mathcal{B}$. \\
 Let us sketch the proof of the third assertion concerning $W$. The continuity is a mere consequence of the limiting behaviour of $s_1$ and $s_2$.  We wish to point out that the boundaries where differentiability issues may appear ($\set{\pr{\phi(i),i}:\ i\in\pp{0,i^*}}$, resp.\ $\{(\varphi^{\frac{\gamma}{\beta(1-\overline a}}(i),i):\ i\in\pp{0,i^*}\}$), are Lebesgue-null sets in $\mathbb{R}^2$. This type of consideration is also valid for $\partial \mathcal{Y}$ (which is 1-dimensional) on which differentiability may only fail at $\frac{\gamma}{\beta}$,  resp.\ $\frac{\gamma}{\beta(1-\overline a)}$. From now on, we leave aside these sets and prove assertion 4.  
On $\overset{\circ}{\mathcal{G}}$, there is nothing to prove (as $W$ is null). On $\mathcal{B}\setminus\pr{\{(\varphi^{\frac{\gamma}{\beta(1-\overline a)}}(i),i):\ i\in\left(0,i^*\right]\}\cup\mathcal{G}}$, one easily computes
\begin{eqnarray*}\partial_sW(s_0,i_0)&=&\frac{1}{\gamma i^*}l_1\pr{s_1,i^*,1-\frac{\gamma}{\beta s_1}}\frac{\pr{\beta s_0-\gamma}s_1}{\pr{\beta s_1-\gamma}s_0}=\tilde{l}_1\pr{s_1,i^*,1-\frac{\gamma}{\beta s_1}}\frac{\beta s_0-\gamma}{\beta s_0}, \\ \partial_iW(s_0,i_0)&=&\frac{1}{\gamma i^*}l_1\pr{s_1,i^*,1-\frac{\gamma}{\beta s_1}}\frac{\beta s_1}{\beta s_1-\gamma}=\tilde{l}_1\pr{s_1,i^*,1-\frac{\gamma}{\beta s_1}}.\end{eqnarray*} On $\mathcal{Y}\setminus\mathcal{B}$, one easily computes 
\begin{equation*}\begin{split}\partial_sW(s_0,i_0)
&=\frac{1}{\gamma\overline{a}}\frac{l_1\pr{s_2,\theta^*-s_2+\frac{\gamma}{\beta (1-\overline a)}\log s_2,\overline a}}{{\theta^*-s_2+\frac{\gamma}{\beta (1-\overline a)}\log s_2}}\pr{1-\frac{\gamma}{\beta s_0}}\\
&=\tilde l_1\pr{s_2,\theta^*-s_2+\frac{\gamma}{\beta (1-\overline a)}\log s_2,\overline a}\frac{\beta s_0-\gamma}{\beta s_0}, \\ 
\partial_iW(s_0,i_0)
&=\frac{1}{\gamma\overline{a}}\frac{l_1\pr{s_2,\theta^*-s_2+\frac{\gamma}{\beta (1-\overline a)}\log s_2,\overline a}}{{\theta^*-s_2+\frac{\gamma}{\beta (1-\overline a)}\log s_2}}=\tilde l_1\pr{s_2,\theta^*-s_2+\frac{\gamma}{\beta (1-\overline a)}\log s_2,\overline a}.\end{split}\end{equation*}
One deduces the following.\begin{enumerate}
\item On $\mathcal{B}\setminus\pr{\{(\varphi^{\frac{\gamma}{\beta(1-\overline a)}}(i),i):\ i\in\left(0,i^*\right]\}\cup\mathcal{G}}$, 
\begin{align*}
&\partial_sW(s_0,i_0)\beta(1-a)s_0i_0-\partial_iW(s_0,i_0)\pr{\beta(1-a)s_0-\gamma}i_0-l_1(s_0,i_0,a)=\\
=&i_0 a\pp{-\tilde l_1\pr{s_1,i^*,1-\frac{\gamma}{\beta s_1}}\pr{\beta s_0-\gamma}+\beta s_0\tilde l_1\pr{s_1,i^*,1-\frac{\gamma}{\beta s_1}}-\gamma\tilde l_1(s_0,i_0,a)}\\&+i_0\pp{\tilde l_1\pr{s_1,i^*,1-\frac{\gamma}{\beta s_1}}\pr{\beta s_0-\gamma}-\beta s_0\tilde l_1\pr{s_1,i^*,1-\frac{\gamma}{\beta s_1}}+\gamma \tilde l_1\pr{s_1,i^*,1-\frac{\gamma}{\beta s_1}}}\\
=&\gamma i_0a\pp{\tilde l_1\pr{s_1,i^*,1-\frac{\gamma}{\beta s_1}}-\tilde{l}_1(s_0,i_0,a)}.
\end{align*}
The conclusion follows using the assumption on $\tilde l_1$ (the coefficient of $a$ is non-positive such that the max is attained for $a=0$ and this maximum is $0$).
\item In the remaining case when $\pr{s_0,i_0}\in\mathcal{Y}\setminus\mathcal{B}$, the conclusion is obtained in exactly the same way; the only difference is that $\tilde l_1\pr{s_1,i^*,1-\frac{\gamma}{\beta s_1}}$ is replaced with $\tilde l_1\pr{s_2,\theta^*-s_2+\frac{\gamma}{\beta (1-\overline a)}\log s_2,\overline a}$ and one concludes using the second inequality in \eqref{GenCond}.
\end{enumerate}
\end{proof}
\begin{remark}
\begin{enumerate}
\item As a consequence, $a^*$ is (an) optimal control for Problem \ref{CtrlProb0} provided the uniqueness result in \cite[Theorem 11]{FrankowskaPlaskacz2000} holds true; please refer to the remarks in the Appendix \ref{A1_1}.  Roughly speaking, with the exception of $\big\{\big(\varphi^{\frac{\gamma}{\beta(1-\overline a)}}(i),i\big):\ 0<i\leq i^*\big\}$, $W$ is a regular (sub)solution of \eqref{HJ} and, thus, it does not exceed the supersolution $V$.  On the other hand, since $a^*$ is admissible, the reverse inequality is valid, hence providing the optimality of $a^*$.
\item The condition \eqref{GenCond} is not only sufficient but actually necessary (among the costs $l_1$ satisfying Assumption \ref{Ass} and \eqref{Assiv}) in order for $a^*$ to be optimal (or, equivalently, for $W$ to be a solution of \eqref{HJ}).
\end{enumerate}
\end{remark}
\begin{example}\label{ExpMain} {\em 
\begin{enumerate}
\item The typical example one has in mind is a multiplicative one\[l_1(s,i,a)=\lambda(s,i)a,\]for some non-negative, uniformly continuous function $\lambda:\mathbb{R}^2\longrightarrow\mathbb{R}_+$ such that 
\begin{itemize}
\item for every $i\in \mathbb{R}_+^*$, the map $\mathbb{R}_+^*\ni s\mapsto \lambda(s,i)$ is non-decreasing (the higher the number of susceptibles,  the higher the price);
\item for every $s\in \mathbb{R}_+^*$,  the map $\mathbb{R}_+^*\ni i\mapsto \frac{\lambda(s,i)}{i}$ is non-increasing.
\end{itemize}
In this case, $\tilde l_1(s,i,a)=\frac{\lambda(s,i)}{\gamma i}$.
Recalling that $s_1\leq s_0$ ($s_1$ is on a trajectory starting at $(s_0,i_0)$ and the $s$ component of such trajectories are non-increasing in time) and that $i_0\leq i^*$ for all $i_0\in\mathcal{Y}$, it follows that
$\tilde l_1\pr{s_1,i^*,1-\frac{\gamma}{\beta s_1}}=\frac{\lambda(s_1,i^*)}{\gamma i^*}\leq \frac{\lambda(s_0,i^*)}{\gamma i^*}\leq \frac{\lambda(s_0,i_0)}{\gamma i_0}$, thus implying the first condition in \eqref{GenCond}. The second condition in \eqref{GenCond} is similar and relies on \eqref{estimi2}.\\
We emphasize, for completeness sake, that the Assumption \ref{Ass} and the condition \eqref{Assiv} hold true for such functions. 
\item The case $l_1(s,i,a):=\lambda a$ studied in \cite{AFG2021+} is a special case of the previous one.  In fact, one can consider general contributions of type $\lambda i^\eta a$, for every $0\leq \eta\leq 1$ (eventually by slightly altering them around $0$ if Lipschitz continuity is desired). 
\end{enumerate}}
\end{example}
\begin{remark}
The non-increasing character of functions $i\mapsto \frac{\lambda(i)}{i}$ amounts to a sub-homogeneity requirement $\lambda(\alpha i)\leq \alpha\lambda(i),\ \forall \alpha\geq 1$. Concave functions $\lambda$ for which $\lambda(0)=0$ satisfy the above-mentioned condition at $i\geq 0$.
\end{remark}
\section{A short journey into discontinuous costs}\label{Section4}
Although the previous arguments give enough novelty with respect to our previous paper \cite{AFG2021+} and with respect to the current literature, let us explain the dual programming approach allowing one to compute the value function in semi-continuous settings. Numerical illustrations and analysis of the algorithm are left for a future work.\\
We wish to point out that, throughout the section,  the cost $l_1$ is assumed to be lower semi-continuous,  bounded from below and control independent, i.e.,  of the form $l_1(s,i)$.  This is a sufficient condition in order to guarantee the equality between strict and relaxed formulations, see, for instance,  \cite[Assumption I, Page 2494]{gaitsgory_quincampoix_09} and discussions following for an idea of more general settings.\footnote{The algorithm developed hereafter can be used for the relaxed problem (where an optimal measure exists).  The assumptions are only sufficient to guarantee that these optimal measures correspond to a control policy.  Of course, convexity in control of $l_1$ is also sufficient.}   
\subsection{The linear programming (LP) approach}
As we have already mentioned before, we deal with an infinite-horizon control problem for which we have chosen the discount parameter $q=0$. Hereafter, we consider a general discount $q> 0$ (the case $q=0$ can be obtained as a limit). To every initial data $\pr{s,i}\in\mathcal{Y}$ and every control $a_0\in\mathcal{A}d$, one associates \begin{itemize}
\item the so-called ($q$-discounted) occupation probability \[\mu^{s_0,i_0,a_0}(ds,di,da)=q\int_0^\infty e^{-qt}\mathbf{1}_{s^{s_0,i_0,a_0}(t)\in ds, i^{s_0,i_0,a_0}(t)\in di, a_0(t)\in da}\,dt;\]
\item the $T>0$ finite-horizon  couple of occupation probability measures
\begin{eqnarray*}
&&\mu_1^{s_0,i_0,a_0}(dt,ds,di,da):=\frac{1}{T}\displaystyle\int_{dt\cap\pp{0,T}}\mathbf{1}_{ s^{s_0,i_0,a_0}(t)\in ds, i^{s_0,i_0,a_0}(t)\in di, a_0(t)\in da}\,dt,\\
&&\mu_2^{s_0,i_0,a_0}(ds,di):=\delta_{\pr{s^{s_0,i_0,a_0}(T),i^{s_0,i_0,a_0}(T)}}(ds,di).
\end{eqnarray*}
\end{itemize}The viability condition of $\mathcal{Y}$ and the dynamic formula for test functions $f\in C^1\pr{\mathbb{R}^2}$ applied to the flow $\pr{s^{s_0,i_0,a_0},i^{s_0,i_0,a_0}}$ yields the following definition of the measure set of constraints \begin{equation}
\label{Theta}
\begin{split}
&\Theta\pr{s_0,i_0}:=\\
&\set{\begin{split}&\mu\in\mathcal{P}\pr{\mathcal{Y}\times \pp{0,\overline a}}:\ \forall f\in C^1\pr{\mathbb{R}^2}\\ &qf(s_0,i_0)+\int_{\mathcal{Y}\times \pp{0,\overline a}}\set{\begin{split}&\pr{-\partial_sf(s,i)}\beta(1-a)si\\&+\partial_if(s,i)\pr{\beta(1-a)s-\gamma}i-qf(s,i)\end{split}}\mu(ds,di,da)=0\end{split}},
\end{split}\end{equation}respectively
\begin{equation}
\label{ThetaT}
\begin{split}
&\Theta_{t_0,T}\pr{s_0,i_0}:=\\
&\set{\begin{split}&\mu:=\pr{\mu_1,\mu_2}\in\mathcal{P}\pr{\pp{t_0,T}\times\mathcal{Y}\times \pp{0,\overline a}}\times \mathcal{P}\pr{\mathcal{Y}}:\ \forall f\in C^1\pr{\mathbb{R}^3}\\ 
&\int_{\mathcal{Y}}f(T,s,i)\mu_2(ds,di)-f(t_0,s_0,i_0)=\\
&=T\int_{\pp{t_0,T}\times\mathcal{Y}\times \pp{0,\overline a}}\set{\begin{split}&\partial_tf(t,s,i)-\partial_sf(t,s,i)\beta(1-a)si\\&+\partial_if(t,s,i)\pr{\beta(1-a)s-\gamma}i\end{split}}\mu_1(dt,ds,di,da)\end{split}}.
\end{split}\end{equation}As usual, given a (subset of a) metric space $S$, $\mathcal{P}(S)$ stands for the family of probability measures on $S$ endowed with the Borel $\sigma$-field.
The reader is referred to \cite{gaitsgory_quincampoix_09} (see also \cite{Goreac_Ivascu_13} for the finite horizon setting or \cite{G2}) for the structure of such sets. In particular (e.g. \cite[Eq. (4.20)]{gaitsgory_quincampoix_09}, see also \cite[Proposition 1]{Goreac_Ivascu_13}), $\Theta\pr{s_0,i_0}$ is the closure of the convex hull of occupation measures as described before and, for every continuous cost $l_1$,  \begin{align*}qV(s_0,i_0)&=q\inf_{a\in\mathcal{A}d}\int_{\mathbb{R}_+}e^{-qt}l_1\pr{s^{s_0,i_0,a}(t),i^{s_0,i_0,a}(t)}dt\\&=\Lambda(s_0,i_0):=\inf_{\mu\in\Theta(s_0,i_0)}\underset{\mathcal{Y}}{\int}l_1(s,i)\mu(ds,di,\pp{0,\overline a}).\end{align*}A similar property holds true for $\Theta_{t_0,T}\pr{s_0,i_0}$ but, for our readers' sake, we will skip unnecessary details.
The dual value (obtained by standard duality algorithms e.g.\ in \cite{G2} or Hamilton-Jacobi methods e.g.\ in \cite[Lemma 3.2]{gaitsgory_quincampoix_09}, \cite[Appendix]{Goreac_Ivascu_13}) can be written in numerous ways:
\begin{enumerate}
\item (cf.\ \cite[Eq (3.1)]{gaitsgory_quincampoix_09})\[\begin{split}\Lambda(s_0,i_0)&=\Lambda^*(s_0,i_0)\\
=&\sup\set{\begin{split}&\eta\in\mathbb{R}:\ \exists f\in C^1\ s.t.\ \ \forall \pr{s,i,a}\in\mathcal{Y}\times \pp{0,\overline a}\\ 
& \left.\begin{split}\eta\leq&-\partial_sf(s,i)\beta(1-a)si+\partial_if(s,i)\pr{\beta(1-a)s-\gamma}i\\&+q\pr{f(s_0,i_0)-f(s,i)}+l_1(s,y)\end{split}\right.
\end{split}}.\end{split}\]
Let us point out that the first equality (see the cited reference) gives the no-gap property.
\item (cf. \cite[Proposition 12]{Serea_Quincampoix_2009} rewriting the original \cite[Lemmas 2.6, 2.7]{barles_jakobsen_02})
\begin{equation}\label{Lambda*2}\Lambda^*(s_0,i_0)
=\sup\set{qf(s_0,i_0):\ f\in \mathbf{\Phi}},
\end{equation}where \begin{equation}\label{F}
\mathbf{\Phi}:=\set{\begin{split} &f\in C^1\ : \ \forall \pr{s,i,a}\in\mathcal{Y}\times \pp{0,\overline a}\\ 
&-\partial_sf(s,i)\beta(1-a)si+\partial_if(s,i)\pr{\beta(1-a)s-\gamma}i-qf(s,i)+l_1(s,i)\leq 0\end{split}}.
\end{equation}
\end{enumerate}
Using the dynamic programming principle for lower semi-continuous functions (given as in \cite[Theorem 14 2.]{Goreac_Ivascu_13} \footnote{The original reference concerns the finite time horizon, but it is easily adaptable to this discounted setting.}), one gets
\begin{proposition}\label{PropDPP}
\begin{enumerate} 
\item If $l_1$ is lower semi-continuous, then the following dynamic programming principle holds true:
\begin{equation}\label{DPP}\Lambda^*(s_0,i_0)=\inf_{\mu=\pr{\mu_1,\mu_2}\in\Theta_{0,T}\pr{s_0,i_0}}\set{\begin{split}&Tq\int_{\pp{0,T}\times\mathbb{R}^2\times \pp{0,\overline a}}e^{-qt}l_1\pr{s,i}\mu_1\pr{dt,ds,di,da}
\\&+e^{-qT}\int_{\mathcal{Y}}\Lambda^*(s,i)\mu_2(ds,di)
\end{split}}.
\end{equation}
\item If $l_1$ is continuous, the following (modified) Dual DPP holds true
\begin{equation}\label{DDPP}\Lambda^*(s_0,i_0)=\inf_{\mu=\pr{\mu_1,\mu_2}\in\Theta_{0,T}\pr{s_0,i_0}}\set{\begin{split}&Tq\int_{\pp{0,T}\times\mathbb{R}^2\times \pp{0,\overline a}}e^{-qt}l_1\pr{s,i}\mu_1\pr{dt,ds,di,da}\\
&+e^{-qT}q\,\underset{f\in\mathbf{\Phi}}{\sup}\int_{\mathcal{Y}}f(s,i)\mu_2(ds,di)
\end{split}}.
\end{equation}

\end{enumerate}
\end{proposition}
The proof is forwardly inspired by the arguments in \cite[Theorem 14 2.]{Goreac_Ivascu_13} and the characterizations of $\Theta$ and $\Theta_{t_0,T}$.  We provide some elements of proof regarding particularly the second formulation in the appendix \ref{A2}.

Furthermore, the compactness of the set $\Theta_{0,T}\pr{s_0,i_0}$ with respect to the weak* convergence of probability measures,
 and the weak* lower semi-continuity   of the functionals 
\begin{eqnarray*}
&&\mu\mapsto \theta(\mu):=Tq\int_{\pp{0,T}\times\mathbb{R}^2\times \pp{0,\overline a}}e^{-qt}l_1\pr{s,i}\mu_1\pr{dt,ds,di,da}+e^{-qT}\int_{\mathcal{Y}}\Lambda^*(s,i)\mu_2(ds,di),\\
&&\mu\mapsto \tilde\theta(\mu):=Tq\int_{\pp{0,T}\times\mathbb{R}^2\times \pp{0,\overline a}}e^{-qt}l_1\pr{s,i}\mu_1\pr{dt,ds,di,da}
+e^{-qT}q\,\underset{f\in\mathbf{\Phi}}{\sup}\int_{\mathcal{Y}}f(s,i)\mu_2(ds,di),
\end{eqnarray*}
 show that the infimum  in  problems  \eqref{DPP} and \eqref{DDPP} is actually a minimum.
\subsection{The dual LP algorithm}
We are now equipped with all the elements allowing one to state the problem as a two-stage linear problem (in the spirit of \cite{Pereira1991} or \cite{Shapiro_2011}).
\begin{problem}
Solve the two stage problem
\begin{itemize}
\item[] 
\begin{equation*}
\overset{\rightarrow}{\mathbb{F}}(s_0,i_0):=\min\big\{e^{-qt}\scal{Tql_1(s,i),\mu_1}+\langle e^{-qT}\overset{\leftarrow}{\mathbb{B}}(s,i),\mu_2\rangle\big\}
\end{equation*}
where the minimum is taken over all measures $\mu=(\mu_1,\mu_2)$ with 
$\supp(\mu)\subseteq\pp{t_0,T}\times\mathcal{Y}\times \pp{0,\overline a}\times \mathcal{Y}$ and such that 
\begin{equation*}\begin{split}
&\scal{f(T,s,i),\mu_2}-f(0,s_0,i_0)\\
&=T\scal{\partial_tf(t,s,i)-\partial_sf(t,s,i)\beta(1-a)si+\partial_if(t,s,i)\pr{\beta(1-a)s-\gamma}i,\mu_1},
\end{split}\end{equation*}
for every $f\in C^1\pr{\mathbb{R}^3}$;
\item[] $$\overset{\leftarrow}{\mathbb{B}}(s,i):=\sup_{f\in C^1} \set{qf(s,i)+\inf_{\pr{s',i'}\in\mathcal{Y}}\pp{\begin{split}&-\partial_sf(s',i')\beta(1-a)s'i'+\partial_if(s',i')\pr{\beta(1-a)s'-\gamma}i'\\
&-qf(s',i')+l_1(s',i')\end{split}}}.$$
\end{itemize}
\end{problem}
\begin{remark}\label{Rem0}
\begin{enumerate}\item The spaces over which one integrates and the integration parameters have been dropped to simplify the presentation. Furthermore, $\scal{\cdot,\cdot}$ denotes the usual duality pairing between measurable functions and measures.
\item Provided a regular candidate $\hat b$ for $\overset{\leftarrow}{\mathbb{B}}$, the forward formulation can be transformed into a dual one \begin{equation*}\begin{split}
&\overset{\leftarrow}{\mathbb{B}}^0(s_0,i_0):=\\&\sup\set{\begin{split}&\eta\in\mathbb{R}:\ \exists\, f\in C^1\pr{\mathbb{R}^2}\ s.t.\ \forall (t,s',i',s,i,a)\in\pp{0,T}\times\mathcal{Y}^2\times\pp{0,\overline a}\\
&\eta\leq \set{\begin{split}&Te^{-qt}\pp{-\partial_sf(s,i)\beta(1-a)si+\partial_if(s,i)\pr{\beta(1-a)s-\gamma}i+ql_1(s,i)}\\
&-qTe^{-qt}f(s,i)+f(s_0,i_0)-e^{-qT}f(s',i')+e^{-qT}\hat{b}(s',i')
\end{split}}\end{split}}.
\end{split}\end{equation*}or, again,
\begin{equation*}
\overset{\leftarrow}{\mathbb{B}}^{00}(s_0,i_0):=\underset{\begin{split}&f\in\mathbf{\Phi},\ qf\geq \hat{b}\end{split}}{\sup}\scal{qf(s,i),\delta_{\pr{s_0,i_0}}}.
\end{equation*}
\end{enumerate}
\end{remark}

As it is usually the case, one projects the problem onto some convenient space that is dense in $C^1$ and for which computations are simplified. Since we are looking for optimal measures that are traditionally best fitted by their moments, we will work with polynomial functions $\mathbb{R}_k\pp{t,s,i}$, where $k\in\mathbb{N}^*$ large enough. For the density of such functions, the reader uses Nachbin's theorem applied to monomials (see \cite[Chapter 1]{Llavona_1986}). As such, we are morally dealing with continuous functions, for which one will use the second form of the dynamic programming principle in Proposition \ref{PropDPP} leading to
\begin{problem}\label{FB1pb}
Solve the two stage problem
\begin{itemize}
\item[] 
\begin{equation*}
\overset{\rightarrow}{\mathbb{F}}(s_0,i_0):=\min\big\{Tq\langle e^{-qt}l_1(s,i),\mu_1\rangle+e^{-qT}\overset{\leftarrow}{\mathbb{B}}_1\pr{\mu_2}\big\}
\end{equation*}
where the minimum is taken over all measures $\mu=(\mu_1,\mu_2)$ with 
$\supp(\mu)\subseteq\pp{t_0,T}\times\mathcal{Y}\times \pp{0,\overline a}\times \mathcal{Y}$ and such that 
\begin{equation*}\begin{split}
&\scal{f(T,s,i),\mu_2}-f(0,s_0,i_0)\\&=T\scal{\partial_tf(t,s,i)-\partial_sf(t,s,i)\beta(1-a)si+\partial_if(t,s,i)\pr{\beta(1-a)s-\gamma}i,\mu_1},
\end{split}\end{equation*}
for every $f\in C^1\pr{\mathbb{R}^3}$;
\item[] $$\overset{\leftarrow}{\mathbb{B}}_1(\mu_2):=\sup_{f\in \mathbf{\Phi}}\scal{qf(s,i),\mu_2}.$$
\end{itemize}
%
\end{problem}

\begin{remark}\label{remms}
The reader is invited to note the formulation $\overset{\leftarrow}{\mathbb{B}}_1$ and the relation with $\overset{\leftarrow}{\mathbb{B}}^{00}$ in Remark \ref{Rem0}, especially for multi-stage problems when the backward step can be repeated in order to provide a lower estimate for the value function at $(s_0,i_0)$.
\end{remark}
Let us now describe an approximation scheme for Problem \ref{FB1pb}.
\begin{enumerate}
\item One begins with recalling the class of polynomials with multi-index of length at most $r\geq 1$, $\mathbb{R}_r\pp{s,i,a}\subset\mathbb{R}\pp{s,i,a}$. We recall that the monomial $s^{\alpha_1}i^{\alpha_2}a^{\alpha_3}$ has the multi-index $\alpha:=\pr{\alpha_1,\alpha_2,\alpha_3}$ with length $\norm{\alpha}=\alpha_1+\alpha_2+\alpha_3$.
\item The cost function $l_1$ has been assumed to be bounded from below and, without loss of generality, one can assume $\inf_{(s,i)\in\mathcal{Y}}l_1(s,i)> 0$ . As such, both the value functions and the polynomials candidate to approximate this value function should be non-negative. 
\item The reader is reminded that the boundary of the set of constraints $\mathcal{Y}$ is characterized by a region $\pr{s,i^*}$ with $s\leq \frac{\gamma}{\beta(1-\overline a)}$ and a region $-s+\frac{\gamma}{\beta\pr{1-\overline a}}\log s= i-i^*+\frac{\gamma}{\beta\pr{1-\overline a}}\log\frac{\gamma}{\beta\pr{1-\overline a}}-\frac{\gamma}{\beta\pr{1-\overline a}}$, with $s\geq \frac{\gamma}{\beta(1-\overline a)}$. For the second region,  by adding the interior, one gets $-s+\frac{\gamma}{\beta\pr{1-\overline a}}\log s\geq i-i^*+\frac{\gamma}{\beta\pr{1-\overline a}}\log\frac{\gamma}{\beta\pr{1-\overline a}}-\frac{\gamma}{\beta\pr{1-\overline a}}$. 
It is perfectly possible to approximate the $\log$ with polynomials\footnote{Please note that $i^*<1$ implies $\frac{\gamma}{\beta (1-\overline a)}\leq \psi\leq \underset{i\rightarrow 0+}{\lim}\psi(i)<1+\frac{\gamma}{\beta (1-\overline a)}$. As a consequence, $\psi(i)=1+l,$ with $l\in (-1,1)$.}, in order to have semi-algebraic spaces. However,  to simplify arguments, we focus on 
\[\big\{\pr{s,i}: 0\leq s\leq \frac{\gamma}{\beta(1-\overline a)},\ 0\leq i\leq i^*\big\}.\] The control component is already in an interval $\pp{0,\overline a}
$. The time component needs no approximation (functions to integrate are always exponential in $t$ and it can be easily shown that the $\mu_1$ measures have $\mathcal{L}eb(dt)$ marginals).
\item Following \cite[Theorem 1.3]{Putinar1993}, we consider the (truncated) quadratic modules\[Q_r:=\set{\begin{split}&p_0^2(s,i,a)+p_s^2(s,i,a)s\pr{\frac{\gamma}{\beta (1-\overline a)}-s}+p_i^2(s,i,a)(i^*-i)i+p_a^2(s,i,a)(\overline a-a)a:\\ &p_0\in \mathbb{R}_{r}[s,i,a],\ p_s,p_i,p_a\in\mathbb{R}_{r-1}[s,i,a]\end{split}}.\]Note that we imposed that the overall index in $Q_r$ does not exceed $2r$.
Approximating $l_1$ comes to picking some $l_1^r\in Q_r$ that is assumed to be fixed throughout the developments. Hereafter, we will just write $l_1$ but we actually use the modifications $l_1^r$. 
\item By abuse of notation, we will still write $Q_r$ (instead of $\tilde Q_r$) for polynomials $p(s,i)$ (that do not depend on $a$). The reason is that one easily obtains $\tilde Q_r=\set{p(s,i,0):\ p\in Q_r}$ thus getting a simple identification.
\item For the \emph{backward step}: given an estimate $\hat \mu^{r,N-1}$ obtained from the previous step, solve the problem 
\begin{equation}\label{BackStep}
\begin{split}
&\overset{\leftarrow}{\mathbb{B}}_1^{r,N}:=\\&\left.\begin{split}\max\Big\{&\langle qp(s,i),\hat\mu_2^{r,N-1}\rangle\ :\ 
p\in Q_r,\\
&\ \partial_sp(s,i)\beta(1-a)si-\partial_ip(s,i)\pr{\beta(1-a)s-\gamma}i+qp(s,i)-l_1(s,i)\in Q_r\Big\}.
\end{split}\right.\end{split}
\end{equation} 
\item Add the optimal $\hat{p}^{r,N}$ to the set of vertices $\Pi^{r,N-1}$ i.e. \[\Pi^{r,N}:=\Pi^{r,N-1}\cup\{\hat{p}^{r,N}\}.\]We emphasize that $\hat{p}^{r,N}$ consists of  a finite number of coefficients describing elements of $Q_r\subset \mathbb{R}_{2r}\pp{s,i,a}$. Since $\Pi^{r,N}\subset \mathbf{\Phi}$ {(the identification of elements $\Pi^{r,N}$ with either vectors of coefficients or polynomials being made with a slight abuse of notation, but cannonically)}, it follows that \begin{align}\label{SuboptimBN}
\hat{b}^{r,N}\pr{\mu_2}:=\underset{p\in\Pi^{r,N}}{\sup}\scal{qp(s,i),\mu_2}\leq \overset{\leftarrow}{\mathbb{B}}_1\pr{{\mu}_2}. 
\end{align}
\item Since $p\in Q_r$, computing $\scal{qp(s,i),\mu_2}$ amounts to compute the moments $m_\alpha$ for a multi-index $\alpha$ whose length does not exceed $2r$.  As such, "remembering" a measure $\mu$ amounts to "remember" the moments \begin{align*}
m_1\pr{\alpha_1,\alpha_2,\alpha_3}&:=q\int_{\mathbb{R}_+\times\mathcal{Y}\times \pp{0,\overline a}}e^{-qt}s^{\alpha_1}i^{\alpha_2}a^{\alpha_3}\mu_1(dt, ds,dy,da),\\ m_2\pr{\alpha_1,\alpha_2}&:=\int_{\mathcal{Y}}s^{\alpha_1}i^{\alpha_2}\mu_2(ds,da),\end{align*} for $0\leq \alpha_1+\alpha_2\leq 2r$ and $\alpha_3\in\set{0,1}$.  In general, if $l_1$ depends on the control $a$,  one adapts to $\norm{\alpha}\leq 2r$ for $\mu_1$ moments.
\item We wish to emphasize that, in our problems, $t$ only appears in a $q$-exponential form. By further assuming that the optimal answer $a$ is a positional strategy $a(s,i)$ (but independent of the time), one can replace $\mu_1$ with the marginal $\overline\mu_1(ds,di,da):=\int_{\mathbb{R}_+}qe^{-qt}\mu_1(dt,ds,di,da)$ in the forward problem.
\item For the \emph{forward step}: given $\hat{b}^{r,N}$ as in $\eqref{SuboptimBN}$,  find the arguments $(\hat{\overline\mu}_1^{r,N},\hat{\mu}_2^{r,N})$ realizing the minimum (with respect to some generated scenarios that will be explained below)
\begin{equation}\label{ForStep}\begin{split}
\min \bigg\{&T\scal{l_1(s,i),\overline\mu_1}+e^{-qT}\hat{b}^{r,N}\pr{\mu_2}\ :\ 
\supp\pr{\overline\mu_1,\mu_2}\subset\mathcal{Y}\times \pp{0,\overline a}\times\mathcal{Y}\textit{ s.t. }\forall\,
\alpha_1+\alpha_2\leq 2r\\ 
&e^{-qT}m_2\pr{\alpha_1,\alpha_2}-s_0^{\alpha_1}i_0^{\alpha_2}=
\\&\qquad=\frac{T}{q}\pr{\begin{split}&-qm_1\pr{\alpha_1,\alpha_2,0}-\alpha_1\beta m_1\pr{\alpha_1,\alpha_2+1,0}+\alpha_1\beta m_1\pr{\alpha_1,\alpha_2+1,1}\\
&+\alpha_2\beta m_1\pr{\alpha_1+1,\alpha_2,0}-\alpha_2\beta m_1\pr{\alpha_1+1,\alpha_2,1}-\alpha_2\gamma m_1\pr{\alpha_1,\alpha_2,0}\end{split}}\bigg\}.
\end{split}
\end{equation}Please note that the above equality on moments is nothing else than the equality constraint on measures in $\Theta_{0,T}$ written for the function $f(t,s,i):=e^{-qt}s^{\alpha_1}i^{\alpha_2}$.\\
The value obtained $\underline{\Gamma}^{r,N}$ cannot exceed $\overline{\Gamma}^{r,N}:=T\langle l_1(s,i),\overline\mu_1^{r,N}\rangle+e^{-qT}\overset{\leftarrow}{\mathbb{B}}_1^{r,N+1}$ (the second-stage being computed as in \eqref{BackStep} with respect to $\hat \mu_2^{r,N}$). The difference between these quantities is used as a stopping criterion.
\item The measures $\mu_2$ can be generated by an a-priori discretization of the compact state space.  To limit the generation of $\mu_2$, one may want to note that the (convex) reachable set starting from $\pr{s_0,i_0}\in\mathcal{Y}$ with viable controls is described by the (trace over $\mathcal{Y}$ of) convex combinations of extremal trajectories (controlled with $0$ for the "upper" branch,  respectively $\overline a$ for the "lower one") i.e.
\begin{equation*}\begin{split}&\mathcal{R}(s_0,i_0):=\\&\set{\begin{split}&\alpha\pr{s_1,i_1}+(1-\alpha)\pr{s_2,i_2} \ :\ 0\leq\alpha\leq 1,\  0\leq i_1,i_2\leq i^*,\ s_1\leq \psi(i_1),\ s_2\leq \psi(i_2),\\
&s_1+i_1=s_0+i_0+\frac{\gamma}{\beta}\log\frac{s_1}{s_0}, \ s_2+i_2=s_0+i_0+\frac{\gamma}{\beta(1-\overline a)}\log\frac{s_2}{s_0}\end{split}}.\end{split}\end{equation*}For the two equalities in $\mathcal{R}$, the reader is invited to recall the computations in \eqref{Comp0} (with minimal control $0$ and maximal control $\overline a$).
Further precision can be added by computing, for $T>0$, $\underline{s_1}^T:=s^{s_0,i_0,0}(T), \ \underline{s_2}^T:=s^{s_0,i_0,\overline a}(T)$ and defining $\mathcal{R}^T(s_0,i_0)$ by further imposing $\underline{s_i}^T\leq s_i,\ i\in\set{1,2}$. This will provide the positions reachable prior to time $T>0$. 
\item Reduction of reachable sets starting from the support of $\mu_2$ is possible  when one uses several stages. Randomization of $T$ such that one fully takes advantage of $\mathcal{R}^T$ is also possible (see \cite{GS2011}).
\item Finally, for each scenario of $\mu_2$ generated, the moments $m_2$ are computed and they provide the family $m_1$ (satisfying the relations in \eqref{ForStep}). Compatibility with the measure structure is then checked for the functionals $m_2$ (cf. \cite[Eq.\ (16)]{Putinar1993} and the comments following this condition).
\end{enumerate}
Finally, let us mention that target problems can be obtained using the same procedure as in \cite[Section 5]{GS2011}.
\section{Appendix}
\subsection{A1}\label{App1}
As mentioned before, the regularity of the value function $V$ (with an impact on the definition of a solution to \eqref{HJ}) makes the object of several papers, usually under some inward pointing qualification assumption (cf.\ \cite{Soner86_1}) or some outward qualification ones (cf.\ \cite{FrankowskaPlaskacz2000}). None of these results apply directly to our problem: either because the analysis partially covers the space (as it is the case when using \cite{Soner86_1}), or because the condition per se is not satisfied everywhere on the boundary (see the analysis in connection with \cite{FrankowskaPlaskacz2000}) but it can be generalized. 

To understand the importance of these considerations, the reader is reminded the following. We have hinted that the value function is a "solution" to \eqref{HJ}. Under reasonable controllability (or qualification) conditions, the solution satisfies (in a viscosity sense), the solution equality in \eqref{HJ} on the interior of the set of constraints $\mathcal{Y}$ and a super-solution ($\leq$ inequality) on the boundary; see \cite[Definition 2.1]{Soner86_1}. Roughly said, the existing results are for either inward pointing conditions (cf.\ \cite{Soner86_1}) or outward pointing conditions (cf.\ \cite{FrankowskaPlaskacz2000}). Our system is of "hybrid" nature boundary. It turns out that, up till $\mathcal{G}$, the boundary $\partial\mathcal{Y}$ will comply with \cite{FrankowskaPlaskacz2000} and on $\mathcal{G}$ (and actually on the larger boundary $\partial\mathcal{Y}\cap\partial\mathcal{B}$), the inward pointing condition is satisfied. It appears in no way surprising that theoretical results can be obtained, but we choose to only hint the connections.
\subsubsection{Connection with \cite{FrankowskaPlaskacz2000}}\label{A1_1}Let us point out some connections with \cite{FrankowskaPlaskacz2000} from which we borrow the notations ($D,M$, etc.) in order to facilitate referral.  
\begin{proposition}
The locally compact set $D:=\mathcal{Y}\setminus\big\{(s_0,i^*):\ \frac{\gamma}{\beta}\leq s_0\leq\frac{\gamma}{\beta(1-\overline a)}\big\}$ is locally-in-time backward invariant for $\pp{0,\overline a}$-valued policies, i.e., for every $\pr{s_0,i_0}\in D$, and every $a\in\mathbb{L}^0\pr{\mathbb{R};\pp{0,\overline a}}$, there exists $t_0>0$ such that $\pr{s^{s_0,i_0,a}(-t),i^{s_0,i_0,a}(-t)}\in D$ for every $t\in\pp{0,t_0}$.  Moreover, $\overline D=\mathcal{Y}$.
\end{proposition}
\begin{proof}
We only need to show something for initial conditions on the active boundary of $\mathcal{Y}$, defined in \eqref{dY+}, that is if $\pr{s_0,i_0}\in\partial\mathcal{Y}^+\cap D$; otherwise, using the continuous dependency of the initial data, and the explicit form of $\mathcal{Y}$, the conclusion is obvious. In this setting, $s_0\geq\frac{\gamma}{\beta(1-\overline a)}$. Then, this inequality is preserved for $t\mapsto s^{s_0,i_0,a}(-t)$. As a consequence, \[\begin{split}&d\pp{i^{s_0,i_0,a}(-t)}=-\pr{\beta(1-a)s^{s_0,i_0,a}(-t)-\gamma}i^{s_0,i_0,a}(-t)\leq 0\end{split}\] which implies $i^{s_0,i_0,a}(-t)\leq i_0< i^*$ and, thus, it follows that $\pr{s^{s_0,i_0,a}(-t),i^{s_0,i_0,a}(-t)}\in \mathcal{Y}$. 
Indeed, setting $(s_1,i_1)=(s^{s_0,i_0,a}(-t),i^{s_0,i_0,a}(-t))$, by taking the control $\tilde{a}(r):=a(-t+r),\ r\geq 0$ , $(s^{s_1,i_1,\tilde a}(r),i^{s_1,i_1,\tilde a}(r))=\pr{s^{s_0,i_0,a}(r-t),i^{s_0,i_0,\tilde a}(r-t)}$ satisfies $i\leq i^*$ for every $r\in[0,t]$,. At time $t$, the trajectory reaches $(s_0,i_0)$ after which a control can be found to stay in $\mathcal{Y}$. Thus, for $(s_1,i_1)$, we have found a control keeping $i\leq i^*$ (which is the definition of $\mathcal{Y}$). On the other hand, $i^{s_0,i_0,a}(-t)<i^*$, thus,  $\pr{s^{s_0,i_0,a}(-t),i^{s_0,i_0,a}(-t)}\in D$.
\end{proof}

Let us set $M:=\mathcal{Y}\setminus D=\big\{\pr{s_0,i^*}:\ \frac{\gamma}{\beta}\leq s_0\leq\frac{\gamma}{\beta(1-\overline a)}\big\}$.  If $\frac{\gamma}{\beta}< s_0<\frac{\gamma}{\beta(1-\overline a)}$, then $\partial\mathcal{Y}^+$ is regular and the outward normals are of type $(0,1)$. Then $\scal{\pr{-\beta s_0i^*,\pr{\beta s_0-\gamma}i^*},\pr{0,1}}>0$ and it follows that the $0$-controlled trajectory does not belong to the paratingent set $P_{\mathcal{Y}}^M(s_0,i^*)$ (for the definition and the relevance of these computations, please refer to \cite[Page 822 and the Remark (2) on page 823]{FrankowskaPlaskacz2000}). This is easily extended at $\pr{\frac{\gamma}{\beta(1-\overline a)},i^*}$.  However,  the same argument no longer works at $\pr{\frac{\gamma}{\beta},i^*}$.  It follows that one cannot directly apply \cite[Theorem 11]{FrankowskaPlaskacz2000} to infer the connection with $\eqref{HJ}$. The attentive reader will have noted that the litigious point $\pr{\frac{\gamma}{\beta},i^*}$ belongs to $\mathcal{G}$ which is the invariance domain with respect to our control system \eqref{SIR} and that the remaining of $\partial\mathcal{Y}$ (i.e. $\pp{0,\frac{\gamma}{\beta}}\times\set{i^*}$ satisfies the inward pointing qualification condition with every control but $0$).

\begin{remark} \begin{enumerate}
\item The generalization of \cite[Theorem 11]{FrankowskaPlaskacz2000} to our present setting can be obtained, under Assumption \ref{Ass} and the hypotheses \eqref{Assiv}  and \eqref{GenCond}, as follows. First, note that our function $W$ in \eqref{CompW} is differentiable almost everywhere and continuous. At the points of differentiability, the subdifferential set is a singleton and, using the computation of the derivatives of $W$ done in the proof of Theorem \ref{ThMain}, it is easily seen that the conditions \cite[Eq. (26), (27)]{FrankowskaPlaskacz2000} are satisfied. This yields \cite[Step 1 in Theorem 11]{FrankowskaPlaskacz2000}, first at every point where $W$ is differentiable, then, using the continuity of $W$, also on the curve where $W$ may not be differentiable. The inequality $W\leq V$ can be proven following the arguments in \cite[Proposition 9]{FrankowskaPlaskacz2000} due to the previous considerations.
\item \cite[Proposition 9]{FrankowskaPlaskacz2000} uses a boundary condition at the final time. In our case, this condition is replaced with $V=0$, on $\mathcal{G}$. Furthermore, in order to guarantee that accumulations do not happen outside $\mathcal{G}$, one requires \begin{equation}
\label{ass+1}
l_1(s,0,0)>0,\ \forall (s,0)\notin\mathcal{G}.
\end{equation}
The need for boundary conditions in order to guarantee comparison results is also illustrated in \cite{BFZ2011}.
\end{enumerate}
\end{remark}
The assumptions in \cite[Proposition 9]{FrankowskaPlaskacz2000} can be directly enforced by slightly modifying the yellow (admissible) set in such a way that $\mathcal{G}$ is the only one affected. The reader will note that the value function is kept the same if one further assumes \begin{equation}\label{ass+}
l_1(s,i,a)=0, \textnormal{ for all }a\in\pp{0,\overline a}\textnormal{ whenever }\pr{s,i}\in\mathcal{G}.
\end{equation} We consider \[\tilde {\mathcal{G}}=\big\{\pr{s,i}\in\mathcal{G}:\ s\geq \tilde{\psi}(i)\big\}, \textnormal{ and }\tilde{Y}=\pr{Y\setminus \mathcal{G}}\cup\tilde{\mathcal{G}},\]where $\big\{\tilde{\psi}(i),i\big)$ describes the curve $\set{\pr{s^{\frac{\gamma}{\beta},i^*,\overline a}(t),i^{\frac{\gamma}{\beta},i^*,\overline a}(t)}:\ t\geq 0}$ (please see Figure \ref{Fig3}). It is easy to see that $\tilde{\psi}(i)=x$ is the solution of $-x+\frac{\gamma}{\beta\pr{1-\overline a}}\log x=i-i^*+\frac{\gamma}{\beta\pr{1-\overline a}}\log\frac{\gamma}{\beta}-\frac{\gamma}{\beta}$, but, this time asking that $x\leq \frac{\gamma}{\beta}$. By computing derivatives in the above definition, one gets \begin{align*}
&\tilde{\psi}'(i)\pp{-1+\frac{\gamma}{\beta(1-\overline a)\tilde{\psi}(i)}}=1\ \iff\ 
\tilde{\psi}'(i)=\frac{1}{-1+\frac{\gamma}{\beta(1-\overline a)\tilde{\psi}(i)}}.
\end{align*}
Please note that this derivative is positive. It is easy to see that 
\begin{enumerate}
\item The $0$-controlled trajectory starting from $\pr{\frac{\gamma}{\beta},i^*}$ leaves $\partial\mathcal{Y}$ (thus completing the paratingent assumption at $\pr{\frac{\gamma}{\beta},i^*}\in M$). The reader is referred to the red line in Figure \ref{Fig3}.
\item The set $\tilde{D}:=\tilde{\mathcal{Y}}\setminus M$ is locally-in-time backward invariant.  To see this, we only need to show that every (backward) trajectory starting at $\pr{\tilde{\psi}(i),i}$ with $i<i^*$ satisfies $s^{\tilde{\psi}(i),i,a}(-t)-\tilde{\psi}\pr{i^{\tilde{\psi}(i),i,a}(-t)}\geq 0$. To this purpose, we drop the superscript and compute\begin{align*}
&d\pp{s(-t)-\tilde{\psi}\pr{i(-t)}}\\
=&\beta(1-a(-t))s(-t)i(-t)+\frac{1}{-1+\frac{\gamma}{\beta(1-\overline a)\tilde{\psi}(i(-t))}}\pr{\beta(1-a(-t))s(-t)-\gamma}i(-t)\\
=&\frac{\gamma i(-t)}{-1+\frac{\gamma}{\beta(1-\overline a)\tilde{\psi}(i(-t))}}\pr{\frac{1-a(-t)}{1-\overline a}\frac{s(-t)}{\tilde{\psi}\pr{i(-t)}}-1}\geq \frac{\gamma i(-t)}{-1+\frac{\gamma}{\beta(1-\overline a)\tilde{\psi}(i(-t))}}\frac{s(-t)-\tilde{\psi}\pr{i(-t)}}{\tilde{\psi}\pr{i(-t)}}.
\end{align*}
It follows that $s(-t)-\tilde{\psi}\pr{i(-t)}$ is non-decreasing, hence it stays non-negative.
\end{enumerate}

\begin{figure}[!t]
\centerline{\includegraphics[width=\columnwidth]{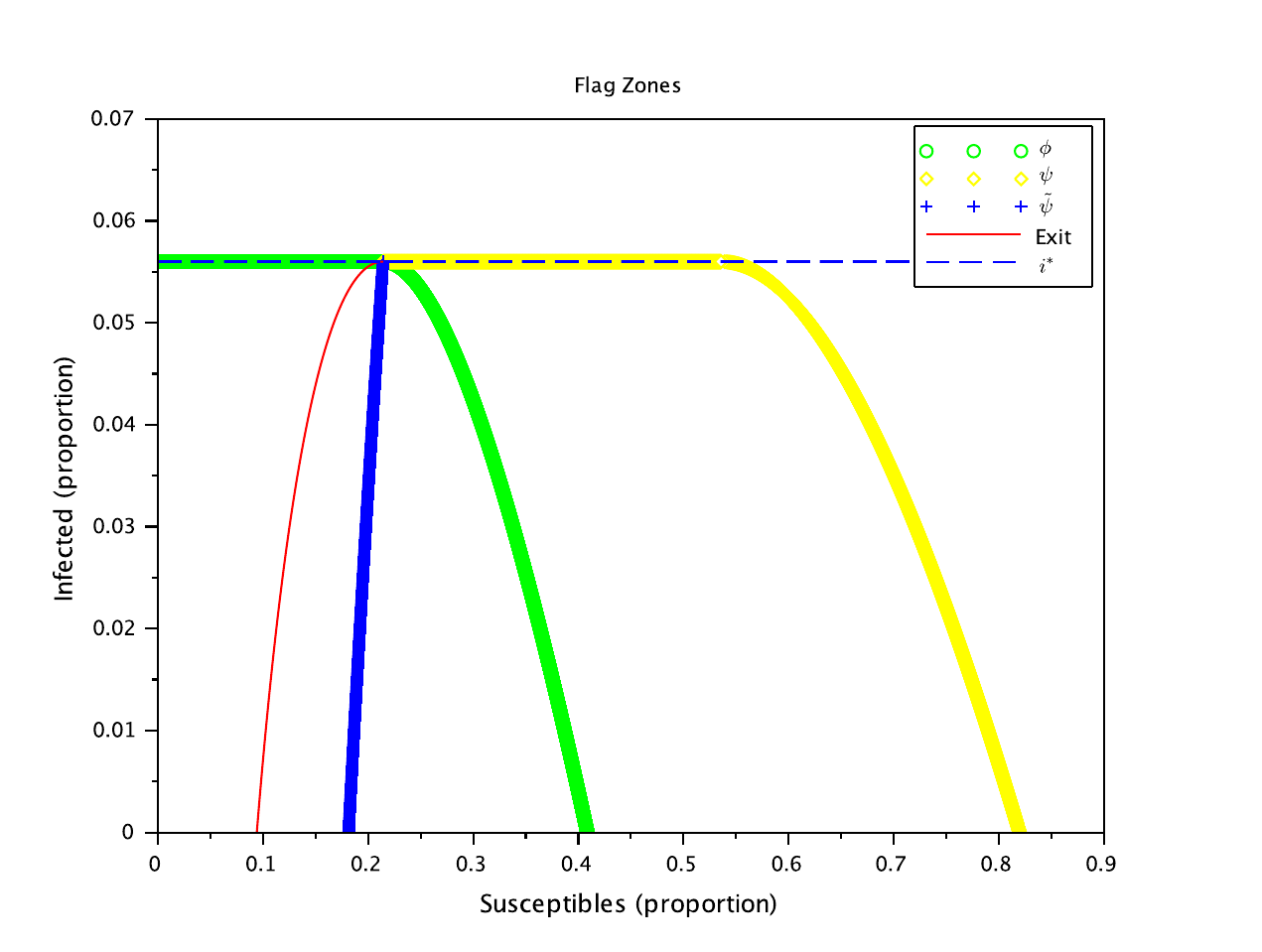}}
\caption{Modified $\tilde{\mathcal{Y}}$ and the outer pointing trajectory at $\frac{\gamma}{\beta}$.}
\label{Fig3}
\end{figure}
{
We do not actually deal with the equation \eqref{HJ} as written, but merely check for this equation in order to avoid the presence of the discount factor $q$.  We focus on the analysis on $\tilde{\mathcal{Y}}$ with the aforementioned assumptions and properties. To be more rigorous, one should define \begin{equation}
\begin{cases}
\tilde J_q\pr{s_0,i_0,a}:=\int_0^\infty e^{-qt}l_1\pr{s^{s_0,i_0,a}(t),i^{s_0,i_0,a}(t),a(t)}dt;\\
V_q\pr{s_0,i_0}:=\underset{a\in\mathcal{A}d}{\inf}\tilde{J}_q\pr{s_0,i_0,a}.
\end{cases}
\end{equation}
The associated equation is \begin{equation}\label{HJq}
qV_q(s,i)+\sup_{a\in \pp{0,\overline{a}}}\set{\beta si(1-a)\partial_s V_q\pr{s,i}-\pr{\beta  (1-a)s-\gamma}i\partial_i V_q(s,i)-l_1(s,i,a)}=0,\textnormal{ on } \tilde{\mathcal{G}}.
\end{equation}
To simplify notations, we set, for $\pr{s,i,p_s,p_i}\in\mathbb{R}^4$,
\begin{equation}
\label{Hamiltonianq}
\begin{cases}
H(s,i,p_s,p_i):=&\underset{a\in \pp{0,\overline{a}}}{\sup}\set{-\beta si(1-a)p_s+\pr{\beta  (1-a)s-\gamma}ip_i-l_1(s,i,a)};\\
H_{out}\pr{s,i,p_s,p_i}:=&\sup\big\lbrace-\beta si(1-a)p_s+\pr{\beta  (1-a)s-\gamma}ip_i-l_1(s,i,a)\ : \\
&\quad \quad a\in \pp{0,\overline{a}},\ 
\pr{-\beta si(1-a),(\beta s(1-a)-\gamma)i}\notin P_{\tilde{\mathcal{G}}}^M\big\rbrace, \ \forall \pr{s,i}\in M.
\end{cases}
\end{equation}
The reader is invited to note that the paratingent condition defining $H_{out}$ and the continuity of the vector field defining the dynamics lead to 
\[H_{out}\pr{s,i,p_s,p_i}:=\sup_{a\in \pp{0,1-\frac{\gamma}{\beta s}}}\set{-\beta si(1-a)p_s+\pr{\beta  (1-a)s-\gamma}ip_i-l_1(s,i,a)},\]for all $\frac{\gamma}{\beta}<s_0<\frac{\gamma}{\beta(1-\overline{a})},\ i=i^*$.
At the corners,  all but one control $\overline{a}$ take the trajectory out of $\tilde{\mathcal{Y}}$ such that $H_{out}=H_{in}$.
With these in mind, the precise understanding of the solution (see \cite[Theorem 6, c)]{FP_1998}) is the following.
\begin{definition}
$W_q$ is a (viscosity) solution to \eqref{HJq} if \begin{enumerate}
\item $W_q$ is a lower semi-continuous bounded function;
\item for every $(s,i)\in \tilde{D}$, and for all $n\in\partial_-W_q(s,i)$, $qW_q(s,i)+H(s,i,-n)=0$;
\item for every $(s,i)=(s,i^*)\in M$, and every $n\in\partial_-W_q(s,i)$, one has $qW_q(s,i)+H(s,i,-n)\geq 0$ and $qW_q(s,i)+H_{out}(s,i,-n)\leq 0$.
\end{enumerate}
\end{definition}
\begin{remark}
\begin{enumerate}
\item We do not insist here neither on the definitions of subdifferentials nor on their link with negative polar cones to the epigraph or to proximal or Mordukovich cones. Instead, the interested reader can take a further look at \cite[Section 2]{FP_1998}.
\item The condition on $H_{out}$ is only needed (see \cite[Lemma 15]{FP_1998}) to guarantee continuity at points in $M$.  Our candidates are shown to satisfy such continuity properties such that $H_{out}$ can be left aside (please see \cite[Proposition 16]{FP_1998} providing the comparison result without the supersolution-like condition on $H_{out}$).
\item Subdiferentials at $\pr{s,i}$ are either approximated by (convex combinations) of gradients at points of differentiability or computed on the lines $\set{(\phi(i),i)\ :\ i\in\pp{0,i^*} }$, 
 respectively on $\set{(\varphi^{\frac{\gamma}{\beta (1-\overline a)}}(i),i)\ :\ i\in\pp{0,i^*} }$ or, eventually,  on $M$. 
\item Along the aforementioned line $(\phi(i),i)$, the function $\tilde{J}_q\pr{s_0,i_0,a^*}=0$ is constant such that the contribution to the subdifferentials is $n=(0,0)$ and there is nothing to prove.  On the contrary, the definition of $s_1$ shows that there is no particular problem of differentiability on $\set{(\varphi^{\frac{\gamma}{\beta (1-\overline a)}}(i),i)\ :\ i\in\pp{0,i^*} }$. The same holds true on $M$, where $\tilde{J}_q\pr{s_0,i^*,a^*}=\frac{1}{\gamma i^*}e^{-\frac{qs_0}{\gamma i^*}}\int_{\frac{\gamma}{\beta}}^{s_0}e^{\frac{qu}{\gamma i^*}}l_1\pr{u,i^*,1-\frac{\gamma}{\beta u}}du$.
\end{enumerate}
\end{remark}
Following similar arguments as for Lemma \ref{CompW},  it follows that $\tilde{J}_q\pr{s_0,i_0,a^*}$ is a (continuous) solution to \eqref{HJq}. Since all the assumptions of \cite[Theorem 6]{FP_1998} are satisfied on $\tilde{\mathcal{G}}$, it follows that $V_q(\cdot)=\tilde J_q\pr{\cdot,a^*}$, hence $a^*$ is optimal for $V_q$.\\
As a by-product, $a^*$ is optimal for $V$ (with $q=0$). Indeed, since $l_1$ is non-negative, it is obvious that $V_q\leq V$.  Using Lebesgue's dominated convergence, it is obvious that $\tilde J(s_0,i_0,a^*)=\lim_{q\rightarrow 0+}\tilde J_q(s_0,i_0,a^*)$. Then $V(s_0,i_0)\geq \tilde J(s_0,i_0,a^*)$ and the optimality follows.\\
When we say that "heuristically" $V$ satisfies the equation, it is because, under the standing assumptions on $l_1$, it is obtained as the limit of continuous solutions to \eqref{HJq}.}
\subsubsection{Connection with \cite{Soner86_1}}\label{A1_2}
Let us assume that, in addition to the Assumption \ref{Ass}, the condition \eqref{Assiv} holds true.  
\begin{proposition}\label{PropB} Whenever $\pr{s_0,i_0}\in \mathcal{B}^{\overline s}$, with $\overline s\in\big[\frac{\gamma}{\beta},\frac{\gamma}{\beta\pr{1-\overline a}}\big]$ and $a\in\mathcal{A}d$ such that $\pr{s^{s_0,i_0,a},i^{s_0,i_0,a}}\in\mathcal{Y}$ on $\mathbb{R}_+$, it holds true that $\pr{s^{s_0,i_0,a},i^{s_0,i_0,a}}\in\mathcal{B}^{\overline s}$ on $\mathbb{R}_+$.
\end{proposition}
\begin{proof}
The argument is quite simple. Nonetheless, we provide the elements for our readers' sake. \\
First, the conclusion is obvious for $\overline s=\frac{\gamma}{\beta}$ as $\mathcal{B}^{\frac{\gamma}{\beta}}=\mathcal{G}$ which is invariant. \\
For $\overline s>\frac{\gamma}{\beta}$, by computing derivatives in the definition of $\varphi^{\overline s}$, it follows that 
\begin{equation}\label{dfisb}
\pr{\varphi^{\overline s}}'(i)\pr{-1+\frac{\gamma}{\beta \varphi^{\overline s}(i)}}=1\Leftrightarrow \pr{\varphi^{\overline s}}'(i)=\frac{\beta \varphi^{\overline s}(i)}{\gamma-\beta \varphi^{\overline s}(i)}. 
\end{equation}
Thus, ${\varphi^{\overline s}}'(i)<0$ as long as $\varphi^{\overline s}(i)>\frac{\gamma}{\beta}$ (which is always true due to the inequality $\overline s>\frac{\gamma}{\beta}$).\\
Computing the derivative in time of $t\mapsto s^{s_0,i_0,a}(t)-\varphi^{\overline s}\pr{i^{s_0,i_0,a}(t)}$ one gets\begin{align*}
\frac{d}{dt}\pr{s^{s_0,i_0,a}(t)-\varphi^{\overline s}\pr{i^{s_0,i_0,a}(t)}}&=\frac{-\gamma i^{s_0,i_0,a}(t)}{\gamma-\beta \varphi^{\overline s}\pr{i^{s_0,i_0,a}(t)}}\pr{\beta(1-a(t))s^{s_0,i_0,a}(t)-\beta \varphi^{\overline s}\pr{i^{s_0,i_0,a}(t)}}\\
&\leq \frac{-\gamma\beta i^{s_0,i_0,a}(t)}{\gamma-\beta \varphi^{\overline s}\pr{i^{s_0,i_0,a}(t)}}\pr{s^{s_0,i_0,a}(t)-\varphi^{\overline s}\pr{i^{s_0,i_0,a}(t)}}.
\end{align*}
The conclusion follows from Gronwall's inequality and the initial condition  $\pr{s_0,i_0}\in\mathcal{B}^{\overline s}$ (which implies $s_0\leq \varphi^{\overline s}(i_0)$).
\end{proof}

Since, on the other hand, every control $a\in\mathcal{A}d$ keeping the trajectory in $\mathcal{B}^{\overline s}$ also keeps it in $\mathcal{Y}$,  it follows that Problem \ref{CtrlProb0} restricted to $\mathcal{B}^{\overline s}$ (with $\overline s\in\big[\frac{\gamma}{\beta},\frac{\gamma}{\beta(1-\overline a)}\big]$) is equivalent with

\begin{problem}\label{prest} For every initial configuration $\pr{s_0,i_0}\in \mathcal{B}^{\overline s}\subset \mathcal{Y}$, 
\begin{itemize}
\item minimize \eqref{Cost0},
\item  under the state constraint \begin{equation}\begin{split}
\label{icuoverlines} \pr{s^{s_0,i_0,a}(t),i^{s_0,i_0,a}(t)}\in\mathcal{B}^{\overline s}\ \iff\ s^{s_0,i_0,a}(t)\leq \varphi^{\overline s}\pr{i^{s_0,i_0,a}(t)},\ \textnormal{a.s. on }\mathbb{R}_+.
\end{split}\end{equation}
\end{itemize}
\end{problem}
\noindent Let us note that the outward  normal unit vector to the "active" part of the boundary of $\mathcal{B}^{\overline s}$, that is $\set{\pr{\varphi^{\overline s}(i),i}:\ i> 0}\cup\pr{\left(0,\overline s\right]\times\set{i^*}}$, with the exception of the singular point $(\bar s,i^*)$ in which one could instead consider the normal cone,
is  $$n_{\pr{s,i}}=\begin{cases}\pr{0,1},&\mbox{if }\pr{s,i}\in(0,\overline s)\times\set{i^*},\\
\dfrac{\pr{1,-\pr{\varphi^{\overline s}}'(i)}}{\sqrt{1+|\pr{\varphi^{\overline s}}'(i)|^2}},&\mbox{if }i\in(0,i^*)\mbox{ and }s=\varphi^{\overline s}(i).
\end{cases}
$$
Using \eqref{dfisb},  we have \begin{equation}\label{Soner}\begin{split}&\scal{n_{\pr{s,i}},\pr{-\beta(1-\overline a)si,\pr{\beta(1-\overline a)s-\gamma}i}}\\&=\begin{cases}\pr{\beta(1-\overline a)s-\gamma}i^*,&\mbox{if }\pr{s,i}\in(0,\overline s)\times\set{i^*}\\
\frac{1}{\sqrt{1+|\pr{\varphi^{\overline s}}'(i)|^2}}\overline a\gamma i\frac{\beta\varphi^{\overline s}(i)}{\gamma-\beta\varphi^{\overline s}(i)},&\mbox{if }i\in(0,i^*),\ s=\varphi^{\overline s}(i).
\end{cases}\end{split}\end{equation}
It is easy to note that we always have $\scal{n_{\pr{s,i}},\pr{-\beta(1-\overline a)si,\pr{\beta(1-\overline a)s-\gamma}i}}\le0$.
Moreover, on the set $(0,\overline s)\times\set{i^*}$ we have 
\[\begin{split}
&\pr{\beta(1-\overline a)s-\gamma}i^*\leq \pr{\beta(1-\overline a)\overline s-\gamma}i^*<0\ \ \iff\ \bar s<\frac{\gamma}{\beta(1-\overline a)}.
\end{split}\]
On the remaining part of the boundary, that is $\{(\varphi^{\overline s}(i),i)\ :\ i\in(0,i^*)\}$, due to the monotonicity of the function $f(x)=\frac{x}{\gamma-x}$ on $(\gamma,+\infty)$ and of $\varphi^{\overline s}$ on $(0,i^*)$, we have 
$$
\frac{\beta\varphi^{\overline s}(i)}{\gamma-\beta\varphi^{\overline s}(i)}\le
\frac{\beta\varphi^{\overline s}(0^+)}{\gamma-\beta\varphi^{\overline s}(0^+)}<0,
$$
where $\displaystyle\varphi^{\overline s}(0^+):=\lim_{i\rightarrow 0+}\varphi^{\overline s}(i)$ which is finite and larger than $\frac{\gamma}{\beta}$; hence, for $i\ge \delta>0$ we have
\[\begin{split}
&\scal{n_{\pr{s,i}},\pr{-\beta(1-\overline a)si,\pr{\beta(1-\overline a)s-\gamma}i}}\\
&=
\frac{1}{\sqrt{1+|\pr{\varphi^{\overline s}}'(i)|^2}}\overline a\gamma i\frac{\beta\varphi^{\overline s}(i)}{\gamma-\beta\varphi^{\overline s}(i)}\le \overline a\gamma \delta\frac{\beta\varphi^{\overline s}(0^+)}{\gamma-\beta\varphi^{\overline s}(0^+)},
\end{split}\]
and the right hand side is a strictly negative constant. 

Second, the reader is invited to note that $\underset{t\rightarrow\infty}\lim s^{\frac{\gamma}{\beta(1-\overline a)},i^*,0}(-t)=\underset{i\rightarrow 0+}\lim\varphi^{\frac{\gamma}{\beta(1-\overline a)}}(i)>\frac{\gamma}{\beta(1-\overline a)}$.  In particular, if $\overline s$ is close to $\frac{\gamma}{\beta(1-\overline a)}$,  one has $\underset{i\rightarrow 0+}\lim\varphi^{\overline s}(i)>\frac{\gamma}{\beta(1-\overline a)}$. It follows that, locally around the boundary, $i^{s_0,i_0,b}>i_0$, so that, by reducing the study to such (viable) sets, i.e., $\mathcal{B}^{\overline s}\setminus\pr{[\frac{\gamma}{\beta(1-\overline a)},+\infty]\times\pp{0,\delta}}$, the estimates in \eqref{Soner} give an actual inward pointing qualification condition on these sets (see \cite[A3 Page 558]{Soner86_1} for details). As a conclusion, on the reunion of these sets first over $\delta>0$ then over $\overline s\nearrow \frac{\gamma}{\beta(1-\overline a)}$, containing the interior of $\mathcal{B}^{ \frac{\gamma}{\beta(1-\overline a)}}$, the value function is bounded and uniformly continuous.

\subsection{A2}\label{A2}
\begin{proof}[Sketch of the proof of Proposition \ref{PropDPP}]
The first equality follows the same reasoning as in \cite[Theorem 14]{Goreac_Ivascu_13} (with the only difference that we are dealing here with a discounted problem) and its proof will be omitted. 
Let us just mention that, following \cite[Theorem 14]{Goreac_Ivascu_13}, the problem can be reduced to continuous $l_1$ (via inf-convolutions) and without state constraints by adding a penalty term involving $nd_{\mathcal{Y}}$ (the distance function) and allowing $n\rightarrow\infty$.  \\
If $l_1$ is continuous, one considers a state penalty\begin{align*}
\begin{cases}
l_1^n(s,i)&:=l_1(s,i)+n\pr{d_{\mathcal{Y}}(s,i)\wedge 1},\\
V_n(s_0,i_0)&:=\underset{a\in\mathbb{L}^0\pr{\mathbb{R};\pp{0,\overline a}}}{\inf}\displaystyle\int_0^\infty e^{-qt}l_1^n\pr{s^{s_0,i_0,a}(t),i^{s_0,i_0,a}(t)}dt.
\end{cases}
\end{align*}
As in \cite[Theorem 14]{Goreac_Ivascu_13},  \begin{align}\label{estim}\Lambda^*=q\ \sup_{n\geq 1}V_n, \textnormal{ on }\mathcal{Y}.
\end{align} 
For this kind of (state unconstrained) problems,  with $n\geq 1$ fixed, using \cite{krylov_00} (see also \cite[Lemmas 2.6 and 2.7]{barles_jakobsen_02} ), it follows that for every $n\geq 1$, for some $C_n>0$ and for every $\varepsilon>0$,  there exists $v^\varepsilon_n\in C^1_b\pr{\mathbb{R}^2}$ (continuously differentiable real functions that are bounded and have bounded first-order derivatives) such that \begin{equation}\label{estim'}
\begin{cases}
\underset{\pr{s,i}\in\mathbb{R}^2}{\sup}\abs{v^{\varepsilon}_n(s,i)-V_n(s,i)}\leq C_n\varepsilon,\\
\underset{\pr{s,i,a}\in\mathbb{R}^2\times\pp{0,\overline a}}{\sup}\set{\begin{split}&-\beta(1-a)si\partial_sv^\varepsilon_n(s,i)+\pr{\beta(1-a)s-\gamma}i\partial_iv^\varepsilon_n(s,i)\\
&-qv^\varepsilon(s,i)+l_1^n(s,i)\end{split}}\leq 0,\\[2.5ex]
C_n\leq C(1+n),
\end{cases}
\end{equation}
where $C>0$ is a constant independent of $n$ and $\varepsilon$. The last inequality, merely stating that $C_n$ is chosen depending on the modulus of continuity of $V_n$ (i.e.  the Lipschitz constants of the dynamics and coefficients) can be found, for instance,  in \cite[Propositions 18, 19]{Goreac_SIAM_2015} in a much more general setting. In particular, since $l_1^n\geq l_1$, it follows that \[v_n^\varepsilon\in\mathbf{\Phi},\ \forall n\geq1,\forall \varepsilon>0.\]\footnote{In the semi-continuous case, $v_n^\varepsilon$ should be penalized by a term $\omega\pr{\frac{1}{n}}:=\frac{1}{q}\underset{\pr{s,i}\in\mathcal{Y}}{\sup}\pr{\underset{\pr{s',i'}\in\mathbb{R}^2}{\inf}\pr{l_1(s,i)+n\abs{\pr{s,i}-\pr{s',i'}}}-l_1(s,i)}$.  However, its convergence to $0$ requires $l_1$ to be continuous.}
A glance at \eqref{Lambda*2} and \eqref{DPP} shows that one only needs to prove the inequality $\leq$ in \eqref{DDPP}. More precisely, it suffices to prove that for all measures $\mu\in\Theta_{0,T}(s_0,i_0)$, \[\int_{\mathcal{Y}}\Lambda^*(s,i)\mu_2(ds,di)\leq q\ \sup_{f\in\mathbf{\Phi}}\int_{\mathcal{Y}}f(s,i)\mu_2(ds,di).\]
For every $n\geq 1$, by considering $\mathbf{\Phi}\ni f_n\pr{\cdot}:=v_n^{\frac{1}{n^2}}$, it follows from \eqref{estim'} that, on $\mathcal{Y}$, $V_n(s,i)\leq f_n(s,i)+C\frac{1+n}{n^2}$. By integrating with respect to $\mu_2$, it follows that
\[\int_{\mathcal{Y}}V_n(s,i)\mu_2(ds,di)\leq \sup_{f\in\mathbf{\Phi}}\int_{\mathcal{Y}}f(s,i)\mu_2(ds,di)+C\frac{1+n}{n+1}.\]The conclusion follows by passing to the limit as $n\rightarrow\infty$ and using \eqref{estim}, the boundedness of $\Lambda^*$ and recalling that $\mu_2$ is a probability measure.
\end{proof}
\bibliographystyle{abbrv}
\bibliography{Bibl_2021}

\def\cprime{$'$}
\begin{thebibliography}{10}

\bibitem{alvarez2020simple}
F.~E. Alvarez, D.~Argente, and F.~Lippi.
\newblock A simple planning problem for covid-19 lockdown.
\newblock Technical report, National Bureau of Economic Research, 2020.

\bibitem{Ames2020}
A.~D. Ames, T.~G. Molnár, A.~W. Singletary, and G.~Orosz.
\newblock Safety-critical control of active interventions for covid-19
  mitigation.
\newblock {\em IEEE Access}, 8:188454--188474, 2020.

\bibitem{anderson1992infectious}
R.~M. Anderson, B.~Anderson, and R.~M. May.
\newblock {\em Infectious diseases of humans: dynamics and control}.
\newblock Oxford university press, 1992.

\bibitem{Angulo2021}
M.~T. Angulo, F.~Castaños, R.~Moreno-Morton, J.~X. Velasco-Hernández, and
  J.~A. Moreno.
\newblock A simple criterion to design optimal non-pharmaceutical interventions
  for mitigating epidemic outbreaks.
\newblock {\em Journal of The Royal Society Interface}, 18(178):20200803, 2021.

\bibitem{AFG_2022_corr}
F.~Avram, L.~Freddi, and D.~Goreac.
\newblock Corrigendum to “{Optimal control of a SIR epidemic with ICU
  constraints and target objectives}”.
\newblock {\em Applied Mathematics and Computation}, 423:127012, 2022.

\bibitem{AFG2021+}
F.~Avram, L.~Freddi, and D.~Goreac.
\newblock Optimal control of a sir epidemic with icu constraints and target
  objectives.
\newblock {\em Applied Mathematics and Computation}, 418:126816, 2022.

\bibitem{barles_jakobsen_02}
G.~Barles and E.~R. Jakobsen.
\newblock On the convergence rate of approximation schemes for
  {H}amilton-{J}acobi-{B}ellman equations.
\newblock {\em ESAIM, Math. Model. Numer. Anal.}, 36(1):M2AN, Math. Model.
  Numer. Anal., 2002.

\bibitem{Behncke}
H.~Behncke.
\newblock Optimal control of deterministic epidemics.
\newblock {\em Optimal control applications and methods}, 21(6):269--285, 2000.

\bibitem{Benders1962}
J.~F. Benders.
\newblock Partitioning procedures for solving mixed-variables programming
  problems.
\newblock {\em Numerische Mathematik}, 4(1):238--252, Dec 1962.

\bibitem{BFZ2011}
O.~Bokanowski, N.~Forcadel, and H.~Zidani.
\newblock Deterministic state-constrained optimal control problems without
  controllability assumptions.
\newblock {\em ESAIM: COCV}, 17(4):995--1015, 2011.

\bibitem{bolzoni2019optimal}
L.~Bolzoni, E.~Bonacini, R.~Della~Marca, and M.~Groppi.
\newblock Optimal control of epidemic size and duration with limited resources.
\newblock {\em Mathematical biosciences}, 315:108232, 2019.

\bibitem{Borkar_Gaitsgory_2005}
V.~Borkar and V.~Gaitsgory.
\newblock On existence of limit occupational measures set of a controlled
  stochastic differential equation.
\newblock {\em SIAM J. Control Optim.}, 44(4):1436--1473 (electronic), 2005.

\bibitem{G2}
R.~Buckdahn, D.~Goreac, and M.~Quincampoix.
\newblock {Existence of Asymptotic Values for Nonexpansive Stochastic Control
  Systems}.
\newblock {\em {Applied Mathematics and Optimization}}, 70(1):1--28, 2014.

\bibitem{esterhuizen2021epidemic}
W.~Esterhuizen, J.~Lévine, and S.~Streif.
\newblock Epidemic management with admissible and robust invariant sets.
\newblock {\em PLOS ONE}, 16(9):1--28, 09 2021.

\bibitem{FP_1998}
H.~Frankowska and S.~Plaskacz.
\newblock Hamilton-{J}acobi equations for infinite horizon control problems
  with state constraints, in {P}roceedings of {I}nternational {C}onference
  ”{C}alculus of {V}ariations and related topics”, {H}aifa, 1998.
\newblock page 97–116, 1998.

\bibitem{FrankowskaPlaskacz2000}
H.~Frankowska and S.~Plaskacz.
\newblock Semicontinuous solutions of {H}amilton–{J}acobi–{B}ellman
  equations with degenerate state constraints.
\newblock {\em J. Math. Anal. Appl.}, 251(2):818--838, 2000.

\bibitem{gaitsgory_06}
V.~Gaitsgory.
\newblock Averaging and near viability of singularly perturbed control systems.
\newblock {\em Journal of Convex Analysis}, 13(2):329--352, 2006.

\bibitem{gaitsgory_quincampoix_09}
V.~Gaitsgory and M.~Quincampoix.
\newblock Linear programming approach to deterministic infinite horizon optimal
  control problems with discouting.
\newblock {\em SIAM J. Control Optim.}, 48(4):2480--2512, 2009.

\bibitem{Goreac_SIAM_2015}
D.~Goreac.
\newblock Asymptotic control for a class of piecewise deterministic markov
  processes associated to temperate viruses.
\newblock {\em SIAM J. Control Optim.}, 53(4):1860--1891, 2015.

\bibitem{Goreac_Ivascu_13}
D.~{Goreac} and C.~{Iva\c{s}cu}.
\newblock {Discontinuous control problems with state constraints: linear
  formulations and dynamic programming principles.}
\newblock {\em {J. Math. Anal. Appl.}}, 402(2):635--647, 2013.

\bibitem{GS2011}
D.~Goreac and O.-S. Serea.
\newblock Mayer and optimal stopping stochastic control problems with
  discontinuous cost.
\newblock {\em J. Math. Anal. Appl.}, 380(1):327--342, 2011.

\bibitem{hansen2011optimal}
E.~Hansen and T.~Day.
\newblock Optimal control of epidemics with limited resources.
\newblock {\em J. Math. Biol.}, 62(3):423--451, 2011.

\bibitem{HWL_2020}
M.~Hohmann, J.~Warrington, and J.~Lygeros.
\newblock A moment and sum-of-squares extension of dual dynamic programming
  with application to nonlinear energy storage problems.
\newblock {\em European Journal of Operational Research}, 283(1):16--32, 2020.

\bibitem{Kantner}
M.~Kantner and T.~Koprucki.
\newblock Beyond just “flattening the curve”: Optimal control of epidemics
  with purely non-pharmaceutical interventions.
\newblock {\em Journal of Mathematics in Industry}, 10(1):1--23, 2020.

\bibitem{kermack1927contribution}
W.~O. Kermack and A.~G. McKendrick.
\newblock A contribution to the mathematical theory of epidemics.
\newblock {\em Proceedings of the Royal Society of London. Series A, Containing
  Papers of a Mathematical and Physical Character}, 115(772):700--721, 1927.

\bibitem{Ketch}
D.~I. Ketcheson.
\newblock Optimal control of an {SIR} epidemic through finite-time
  non-pharmaceutical intervention.
\newblock {\em J. Math. Biol.}, 83(1):Paper No. 7, 21, 2021.

\bibitem{Kruse}
T.~Kruse and P.~Strack.
\newblock Optimal control of an epidemic through social distancing.
\newblock {\em Cowles Foundation Discussion Paper No. 2229, Available at SSRN:
  https://ssrn.com/abstract=3583186 or http://dx.doi.org/10.2139/ssrn.3583186},
  April 23, 2020.

\bibitem{krylov_00}
N.~V. Krylov.
\newblock On the rate of convergence of finite-difference approximations for
  {B}ellman's equations with variable coefficients.
\newblock {\em Probab. Theory Related Fields}, 117(1):1--16, 2000.

\bibitem{LHT_2008}
J.~B. Lasserre, D.~Henrion, C.~Prieur, and E.~Trélat.
\newblock Nonlinear optimal control via occupation measures and
  lmi-relaxations.
\newblock {\em SIAM Journal on Control and Optimization}, 47(4):1643--1666,
  2008.

\bibitem{Llavona_1986}
J.~G. Llavona, editor.
\newblock {\em Approximation of Continuously Differentiable Functions}, volume
  130 of {\em North-Holland Mathematics Studies}.
\newblock North-Holland, 1986.

\bibitem{Mart}
M.~Martcheva.
\newblock {\em An introduction to mathematical epidemiology}, volume~61.
\newblock Springer, 2015.

\bibitem{Miclo}
L.~Miclo, D.~Spiro, and J.~Weibull.
\newblock Optimal epidemic suppression under an icu constraint.
\newblock {\em arXiv preprint arXiv:2005.01327}, 2020.

\bibitem{Pereira1991}
M.~V.~F. Pereira and L.~M. V.~G. Pinto.
\newblock Multi-stage stochastic optimization applied to energy planning.
\newblock {\em Mathematical Programming}, 52(1):359--375, May 1991.

\bibitem{Putinar1993}
M.~Putinar.
\newblock Positive polynomials on compact semi-algebraic sets.
\newblock {\em Indiana University Mathematics Journal}, 42(3):969--984, 1993.

\bibitem{Serea_Quincampoix_2009}
M.~Quincampoix and O.~S. Serea.
\newblock The problem of optimal control with reflection studied through a
  linear optimization problem stated on occupational measures.
\newblock {\em Nonlinear Analysis, Theory, Methods and Applications},
  72(6):2803--2815, 2009.

\bibitem{ester2021}
P.~Sauerteig, W.~Esterhuizen, M.~Wilson, T.~K.~S. Ritschel, K.~Worthmann, and
  S.~Streif.
\newblock Model predictive control tailored to epidemic models,
  arxiv:2111.06688, 2021.

\bibitem{Shapiro_2011}
A.~Shapiro.
\newblock Analysis of stochastic dual dynamic programming method.
\newblock {\em European Journal of Operational Research}, 209(1):63--72, 2011.

\bibitem{Soner86_1}
H.~M. Soner.
\newblock Optimal control with state-space constraint. {I}.
\newblock {\em SIAM J. Control Optim.}, 24(6):552--561, 1986.

\end{thebibliography}

\end{document}